\documentclass[a4paper, 11pt]{extarticle}
\usepackage[utf8]{inputenc}
\usepackage[english]{babel}
\usepackage{amsmath, amscd, amsthm, amstext, amssymb, verbatim, version, mathrsfs, bbm, stmaryrd, tensor}
\usepackage[all, cmtip]{xy}
\usepackage{dsfont}
\usepackage{geometry}

\title{Generators in formal deformations of categories}
\author{Anthony Blanc, Ludmil Katzarkov, Pranav Pandit}
\date{\today}

\theoremstyle{plain}
\newtheorem{prop}{Proposition}[section]
\newtheorem{cor}[prop]{Corollary}
\newtheorem{lem}[prop]{Lemma}
\newtheorem{theo}[prop]{Theorem}

\theoremstyle{definition}
\newtheorem{df}[prop]{Definition}
\newtheorem{rema}[prop]{Remark}
\newtheorem{ex}[prop]{Example}
\newtheorem{nota}[prop]{Notation}
\newtheorem{const}[prop]{Construction}

\newcommand{\Z}{\mathbb{Z}}

\newcommand{\s}{\mathbb{S}}

\newcommand{\M}{\mathcal{M}}

\newcommand{\A}{\mathbf{A}}
\newcommand{\B}{\mathbf{B}}

\newcommand{\deth}{\mathfrak{D}}
\newcommand{\U}{\mbb{U}}
\newcommand{\V}{\mbb{V}}

\newcommand{\mfr}{\mathfrak{m}}
\newcommand{\tfr}{\mathfrak{t}}
\newcommand{\ccal}{\mcal{C}}
\newcommand{\dal}{\mcal{D}}
\newcommand{\ccalp}{\mcal{C}'}
\newcommand{\ccalpp}{\mcal{C}''}
\newcommand{\acal}{\mcal{A}}
\newcommand{\ecal}{\mcal{E}}

\newcommand{\E}{\mathbb{E}}
\newcommand{\xcal}{\mathcal{X}}

\newcommand{\lmo}{\longrightarrow}
\newcommand{\lmos}[1]{\stackrel{#1} {\longrightarrow}}
\newcommand{\mo}{\rightarrow}
\newcommand{\mos}[1]{\stackrel{#1} {\rightarrow}}

\def\dar[#1]{\ar@<2pt>[#1]\ar@<-2pt>[#1]}

\newcommand{\mbb}[1]{\mathbb{#1}}

\newcommand{\mcal}[1]{\mathcal{#1}}

\newcommand{\mrm}[1]{\mathrm{#1}}

\newcommand{\del}{\Delta}

\newcommand{\te}{\otimes}
\newcommand{\tec}{\widehat{\otimes}}

\newcommand{\unit}{\mathbbm{1}}

\newcommand{\spec}{\mathrm{Spec}}

\newcommand{\ei}{\mbb{E}_\infty}
\newcommand{\eu}{\mbb{E}_1}
\newcommand{\eo}{\mbb{E}_1}
\newcommand{\et}{\mbb{E}_2}
\newcommand{\en}{\mbb{E}_n}
\newcommand{\mon}{\mathrm{Mon}}
\newcommand{\mongp}{\mathrm{Mon}^{\mrm{gp}}}
\newcommand{\monem}{\mathrm{Mon}_{\E_m}}
\newcommand{\monei}{\mathrm{Mon}_{\E_\oo}}
\newcommand{\monemgp}{\mathrm{Mon}^{\mrm{gp}}_{\E_m}}
\newcommand{\moneigp}{\mathrm{Mon}^{\mrm{gp}}_{\E_\oo}}

\newcommand{\alg}{\mathrm{Alg}}

\newcommand{\free}{\mrm{Free}}
\newcommand{\freea}{\mrm{Free}^{\mrm{aug}}} 

\newcommand{\oo}{\infty}

\newcommand{\hh}{\mrm{HH}}

\newcommand{\id}{\mrm{id}}
\newcommand{\op}{\mrm{op}}
\newcommand{\ev}{\mrm{ev}}

\newcommand{\ind}{\mrm{Ind}}
\newcommand{\colim}{\mrm{colim}}

\newcommand{\End}{\mrm{End}}
\newcommand{\uend}{\underline{\mathrm{End}}}

\newcommand{\map}{\mrm{Map}}
\newcommand{\umap}{\underline{\mrm{Map}}}

\newcommand{\fun}{\mrm{Fun}}
\newcommand{\funl}{\mrm{Fun}^\mrm{L}}

\newcommand{\extg}{\mrm{Ext}}

\newcommand{\modk}{\mrm{Mod}_k}

\newcommand{\dgklp}{Dg(k)}
\newcommand{\dgkc}{Dg^{c}(k)}

\newcommand{\nerdglp}{\mrm{N}^\mrm{L}_{dg}}
\newcommand{\nerdgc}{\mrm{N}^c_{dg}}

\newcommand{\catinf}{\mcal{C}\mrm{at}_\infty}
\newcommand{\catbig}{\widehat{\mcal{C}\mrm{at}}_\infty}
\newcommand{\catbigte}{\widehat{\mcal{C}\mrm{at}}^\times_\infty}

\newcommand{\prl}{\mrm{Pr}^{\mrm{L}}}
\newcommand{\prlte}{\mrm{Pr}^{\mrm{L}, \te}}

\newcommand{\prrte}{\mrm{Pr}^{\mrm{R}, \te}}
\newcommand{\prlc}{\mrm{Pr}^L_\omega}
\newcommand{\prlk}{\mrm{Pr}^{\mrm{L}}_k}
\newcommand{\prlkc}{\mrm{Pr}^{\mrm{L}}_{\omega,k}}
\newcommand{\prlkte}{\mrm{Pr}^{\mrm{L}, \te}_k}
\newcommand{\prlkp}{(\mrm{Pr}^{\mrm{L}}_k)_{\modk/}}

\newcommand{\pre}{\mcal{P}}

\newcommand{\catk}{\mrm{Pr}^{\mrm{L}}_k}

\newcommand{\rcat}{\mrm{R}\mcal{C}\mrm{at}}

\newcommand{\rcatk}{\mrm{R}\mcal{C}\mrm{at}_k}
\newcommand{\rcatkp}{\mrm{R}\mcal{C}\mrm{at}_k^*}

\newcommand{\spaces}{\mcal{S}}
\newcommand{\bspaces}{\widehat{\mcal{S}}}

\newcommand{\Sp}{\mrm{Sp}}

\newcommand{\algenk}{\mrm{Alg}_k^{(n)}}
\newcommand{\algenka}{\mrm{Alg}_{k}^{(n),\mrm{aug}}}
\newcommand{\algenkah}{\mrm{Alg}_{k}^{(n),\mrm{aug}, \heartsuit}}
\newcommand{\algenkart}{\mrm{Alg}_{k}^{(n),\mrm{art}}}

\newcommand{\algk}{\mrm{Alg}_k}
\newcommand{\algka}{\mrm{Alg}_{k}^{\mrm{aug}}}
\newcommand{\algkart}{\mrm{Alg}_{k}^{\mrm{art}}}
\newcommand{\algkartte}{\mrm{Alg}_{k}^{\mrm{art}, \te}}

\newcommand{\algetk}{\mrm{Alg}_k^{(2)}}
\newcommand{\algetka}{\mrm{Alg}_{k}^{(2),\mrm{aug}}}
\newcommand{\algetkart}{\mrm{Alg}_{k}^{(2),\mrm{art}}}

\newcommand{\calgka}{\mrm{CAlg}_k^{\mrm{aug}}}
\newcommand{\calgkah}{\mrm{CAlg}_k^{\mrm{aug}, \heartsuit}}
\newcommand{\calgkart}{\mrm{CAlg}_k^{\mrm{art}}}

\newcommand{\rmod}{\mrm{RMod}}
\newcommand{\lmod}{\mrm{LMod}}
\newcommand{\Mod}{\mrm{Mod}}

\newcommand{\catdef}{\mrm{\mathcal{C}atDef}}
\newcommand{\catdefc}{\mrm{\mathcal{C}atDef}^c}

\newcommand{\catdefa}{\mrm{\mathcal{C}atDef}^\wedge}
\newcommand{\objdef}{\mrm{ObjDef}}

\newcommand{\objdefet}{\mrm{ObjDef}^{\et}}
\newcommand{\objdefeta}{\mrm{ObjDef}^{\et, \wedge}}

\newcommand{\objdefa}{\mrm{ObjDef}^\wedge}

\newcommand{\algdef}{\mrm{AlgDef}}

\newcommand{\algdefa}{\mrm{AlgDef}^{\wedge}}

\newcommand{\simdef}{\mathcal{S}\mrm{imDef}}
\newcommand{\simdefc}{\mathcal{S}\mrm{imDef}^c}
\newcommand{\simdefa}{\mathcal{S}\mrm{imDef}^\wedge}

\newcommand{\deform}{\mrm{Deform}}

\newcommand{\indcoh}{\mrm{QCoh}^!}

\newcommand{\rindcoh}{\mrm{QCoh}^!_R}
\newcommand{\rindcohp}{\mrm{QCoh}^!_{R,\ast}}

\newcommand{\indcohte}{\mrm{QCoh}^{!, \te}}

\newcommand{\rindcohte}{\mrm{QCoh}^{!, \te}_R}
\newcommand{\rindcohtep}{\mrm{QCoh}^{!, \te}_{R,\ast}}

\newcommand{\aut}{\mrm{Aut}}

\newcommand{\auts}{\mrm{Aut}^!}

\newcommand{\kfor}{k[[ t ]]}

\newcommand{\fmp}{\mrm{Moduli}_k}

\newcommand{\fmpeo}{\mrm{Moduli}_k^{\eo}}
\newcommand{\eofmp}{\mrm{Moduli}_k^{\eo}}
\newcommand{\eofmpte}{\mrm{Moduli}_k^{\eo, \times}}

\newcommand{\fmpet}{\mrm{Moduli}^{\mbb{E}_2}_k}
\newcommand{\fmpette}{\mrm{Moduli}^{\mbb{E}_2, \times}_k}

\newcommand{\fmpen}{\mrm{Moduli}^{\en}_k}

\newcommand{\fmpa}{\mrm{Moduli}_{\A}}

\newcommand{\tone}{\mathfrak{t}^{(1)}}

\newcommand{\proxma}{\mrm{Prox}(m)_\A}
\newcommand{\prox}{\mrm{Prox}}

\newcommand{\spf}{\mrm{Spf}}

\newcommand{\du}{\mathfrak{D}^{(1)}}
\newcommand{\dt}{\mathfrak{D}^{(2)}}

\newcommand{\dn}{\mathfrak{D}^{(n)}}

\newcommand{\Aart}{\mathbf{A}^{\mrm{art}}}

\usepackage[pdftex]{hyperref}
\hypersetup{colorlinks=true, citecolor=black, linkcolor=black}

\begin{document}

\maketitle

\begin{abstract}
In this paper we use the theory of formal moduli problems developed by Lurie in order to study the space of formal deformations of a $k$-linear $\oo$-category for a field $k$. Our main result states that if $\ccal$ is a $k$-linear $\oo$-category which has a compact generator whose groups of self extensions vanish for sufficiently high positive degrees, then every formal deformation of $\ccal$ has zero curvature and moreover admits a compact generator.
\end{abstract}

\tableofcontents

\newpage

\section{Introduction}

Let $k$ be a field of characteristic zero and let $B$ be an associative $k$-algebra. It is known since the work of Gerstenhaber that the formal deformations of $B$ up to isomorphism are in bijection with the set of solutions of the Maurer--Cartan equation in Hochschild cohomology up to gauge equivalence. In their seminal work \cite{kontsoibook} Kontsevich and Soibelman have pionereed the approach of studying deformations of dg-algebras and dg-categories using differential graded Lie algebras. They show that if $B$ is a connective differential graded algebra over $k$, its deformations as a dg-algebra up to quasi-isomorphism are controlled by the truncated Hochschild cohomology complex viewed as a differential graded Lie algebra.

Given a $k$-linear dg-category $\ccal$, it has been a long standing problem to give the precise relation between the dg-category deformations of $\ccal$ up to Morita equivalence and its Hochschild cohomology complex. In the case of first order infinitesimal deformations, each Hochschild $2$-cocycle corresponds to a curved $A_\oo$-deformation of $\ccal$ (see \cite{lowenhhcha}). In \cite{kellow}, B. Keller and W. Lowen gave examples of curved deformations which are not Morita equivalent to uncurved deformations. In \cite{lowvdbcurv}, W. Lowen and M. Van den Bergh study deformations of a dg-algebra $B$ up to torsion Morita equivalence and show that when $B$ is cohomologically bounded above, the set of formal deformations of $B$ up to torsion Morita equivalence is in bijection with the set of Maurer--Cartan solutions up to gauge equivalence. In order to study the whole space of deformations of a category up to Morita equivalence and not just merely its set of connected components, and also to exhibit the precise relation between this space and Hochschild cohomology, it is useful to place ourselves in the setting of derived algebraic geometry and more precisely in the language of \emph{formal moduli problems} developed by Jacob Lurie in \cite{dagX}. 
\\

\textbf{Formal moduli problems.} Deformation theory was first formalized through the notion of functors on artinian rings satisfying the Schlessinger conditions. Formal moduli problems form a derived and higher analog of these functors. We refer the reader to \cite{dagX} for an introduction to formal moduli problems and their classification. Roughly they are functors from certain artinian $\ei$-algebras\footnote{In characteristic zero, we can also think of these as functors on artinian commutative differential graded algebras.} to spaces satisfying the derived analogs of the Schlessinger conditions. The main result states that formal moduli problems are classified by differential graded Lie algebras (over a field of characteristic zero). In loc. cit. Lurie studies noncommutative analogs of formal moduli problems known as formal $\en$-moduli problems for $n\geq 1$, which are functors from artinian $\en$-algebras to spaces satisfying the derived Schlessinger conditions, where $\en$ is the $\oo$-operad of little $n$-cubes. In this context, the formal $\en$-moduli problems are classified by augmented $\en$-algebras. It is often useful to find the minimum $n$ for which a moduli problem is defined, thereby obtaining the true algebraic structure on the tangent complex (see \cite{francisen} for a discussion of these ideas).   
\\

\textbf{Formal moduli of a category after Lurie.}
Let $k$ be a field of arbitrary characteristic. In loc. cit. Lurie studies deformations of a presentable stable $k$-linear $\oo$-category\footnote{As a matter of facility for citations, from now on we study deformations of linear $\oo$-categories instead of differential graded categories. See the work of L. Cohn \cite{cohndgst} for a comparison between the two theories, recalled in §\ref{sslincat}.} $\ccal$ up to equivalence\footnote{Recall that the notion of equivalence between presentable $\oo$-categories corresponds to the notion of Morita equivalence between the corresponding small subcategories of compact objects. See §\ref{sslincat}.} by looking at the functor
$$\catdef_\ccal: \algetkart\mo \spaces$$
from the $\oo$-category of artinian $\et$-algebras to the $\oo$-category of spaces, classifying deformations of $\ccal$. That is, for an artinian $\et$-algebra $A$, the space $\catdef_\ccal(A)$ classifies pairs $(\ccal_A, u)$ with $\ccal$ a right $A$-linear $\oo$-category and $u:\ccal_A\te_A k \simeq \ccal$ an equivalence. In general, the functor $\catdef_\ccal$ does not satisfy the derived Schlessinger conditions, the moral reason for this being the existence of curved deformations (see Example \ref{exarigid}). When $\ccal$ is compactly generated\footnote{An $\oo$-category is compactly generated if it is generated under filtered colimits by its subcategory of compact objects (as for example the derived $\oo$-category of quasi-coherent sheaves on a finite type scheme).}, the subfunctor $\catdefc_\ccal\subseteq \catdef_\ccal$ spanned by compactly generated deformations is better behaved and closer to satisfying the Schlessinger conditions. We denote by $\hh^*(\ccal)$ the $\et$-algebra given by the center of $\ccal$ whose homotopy groups are the Hochschild cohomology groups of $\ccal$. Lurie constructs a natural transformation 
$$\theta_\ccal : \catdefc_\ccal\mo \Psi_{\hh^*(\ccal)}=\catdefa_\ccal$$
from $\catdefc_\ccal$ to the formal moduli problem controlled by the augmented $\et$-algebra $k\oplus \hh^*(\ccal)$. He shows that under certain reasonable boundedness assumptions on $\ccal$ the map $\theta_\ccal$ is close to being a homotopy equivalence: it induces an isomorphism on $\pi_i$ for $i\geq 1$ (at any basepoint) and an injection on $\pi_0$. The formal moduli $\catdefa_\ccal$ can be regarded as a substitute for the space of curved $A_\oo$-deformations and the map $\theta_\ccal$ as the embedding from uncurved deformations to curved deformations. By the universal property of the center, for each artinian $\et$-algebra $A$, the space $\catdefa_\ccal(A)$ is given by the space of linear actions of the $\et$-algebra $\dt(A)$ on the $\oo$-category $\ccal$, where $\dt(A)$ is the $\et$-Koszul dual of $A$.

\subsection{Main results}

Our main results provide a new understanding of the behavior of the map $\theta_\ccal$ in the case of formal deformations, that is, over the adic ring $\kfor$ of formal power series. Namely, we give conditions on $\ccal$ under which the map $\theta_\ccal$ induces a homotopy equivalence on formal deformations, therefore providing the missing step to relate deformations of categories to Hochschild cohomology. It can also be viewed as a solution to the curvature problem of \cite{lowvdbcurv}, where similar computations are performed.

Let $k$ be a field and let $\ccal$ be compactly generated $k$-linear $\oo$-category. The space of compactly generated deformations of $\ccal$ is the space 
$$\catdefc_\ccal(\kfor)\simeq \lim_i \catdefc_\ccal(k[t]/t^i).$$
That is, a point in this space is essentially given by a formal family $\{\ccal_i\}_{i\geq 1}$ where each $\ccal_i$ is a compactly generated deformation of $\ccal$ over $k[t]/t^i$ and $\ccal_{i+1}\te_{k[t]/t^{i+1}} k[t]/t^i \simeq \ccal_i$ is an equivalence, plus higher coherences. In this case, the space $\catdefa_\ccal(\kfor)$ is given by the space of linear actions of the $\et$-algebra $\dt(\kfor)\simeq k[u]$ on the $\oo$-category $\ccal$, where $k[u]$ is the free graded commutative algebra on a generator $u$ in cohomological degree $2$ viewed as an $\et$-algebra.

\begin{theo}\label{thmmainint} \emph{(See Theorem \ref{thmcatdeffor}).} --- 
Let $k$ be a field and let $\ccal$ be a compactly generated $k$-linear $\oo$-category. Suppose that $\ccal$ admits a single compact generator $E$ such that $\extg_\ccal^m(E,E)=0$ for $m\gg 0$. Then the natural transformation $\theta_\ccal$ induces a homotopy equivalence 
$$\catdefc_\ccal (\kfor)\simeq \catdefa_\ccal(\kfor)=  \{k[u]-\textrm{linear structures on }\ccal\}.$$  
\end{theo}

The idea of the proof is to exhibit another compact generator of $\ccal$ on which $k[u]$ acts trivially, which implies that it is unobstructed. For this it suffices to kill the multiplication by $u$ on the generator $E$ by forming its cocone. As a by product of the proof of Theorem \ref{thmcatdeffor} we obtain the following existence result for compact generators.

\begin{theo}\label{thmgen}\emph{(See Theorem \ref{thmgenfordef}).} --- 
In the situation of Theorem \ref{thmmainint}, let $\ccal_t=\{\ccal_i\}_i$ be any compactly generated formal deformation of $\ccal$. Then there exists a formal family of compact generators $\{E_i\}_i$ of the family $\{\ccal_i\}_i$ which satisfies $\extg_{\ccal_i}^m(E_i,E_i)=0$ for $m\gg 0$. 
\end{theo}

We emphasize that our method in the proof of Theorems \ref{thmmainint} and \ref{thmgen} relies on the fact that the $\et$-Koszul dual of the algebra $\kfor$ of formal power series (viewed as an $\et$-algebra) is given by the free graded associative (also commutative) algebra $k[u]$ with $u$ in degree $2$ (viewed as an $\et$-algebra). The argument cannot be adapted verbatim to first order infinitesimal deformations where the $\et$-algebra $\dt(k[t]/t^2)$ is given by the more complicated free $\et$-algebra $\free_{\et} (k[-2])$ on a generator in degree $2$. 

F. Petit has proven in \cite{petitdgaff} the existence of compact generators in derived categories of $DQ$-modules. We expect to give an alternative proof of this fact using Theorem \ref{thmgen}. 
\\

On a more fundamental level, we prove a new result concerning the loop functor on the $\oo$-category of formal moduli problems over a deformation context (Proposition \ref{propfmpgp}). Moreover it implies a new description of the formal moduli associated to any $m$-proximate formal moduli (Corollary \ref{corassfmp}). As a consequence, we obtain the following description of the formal moduli $\catdefa_\ccal$ in terms of $!$-group actions (which also appeared in \cite[Prop 5.3.3.4]{preygelth}).

\begin{prop} \emph{(See Corollary \ref{corsgroupact}).} --- 
 Let $k$ be a field and let $\ccal$ be a $k$-linear $\oo$-category. Then the formal $\et$-moduli problem $\catdefa_\ccal$ is given by the formula 
 $$\catdefa_\ccal(A)\simeq \{\textrm{Left actions of }\indcoh(\Omega \spf(A))^\te \textrm{ on } \ccal\}$$
 where $\indcoh(\Omega \spf(A))$ is the $\oo$-category of Ind-coherent sheaves on $\Omega\spf(A)$ endowed with the convolution monoidal structure. 
 \end{prop}

\subsection*{Acknowledgments} A special acknowledgment and many thanks are due to Bertrand Toën for generously sharing his ideas on the subject matter of this paper and for many enlightening discussions. We would like to thank Mohammed Abouzaid, Alexander Efimov, David Gepner, Benjamin Hennion, Julian Holstein, Maxim Kontsevich, Wendy Lowen, Tony Pantev, François Petit, Marco Robalo, Pavel Safronov, Carlos Simpson, Michel Van den Bergh and Gabriele Vezzosi for useful discussions related to the subject matter of this paper. 

The second and third named authors were partially supported by ERC Gemis grant. The second named author was supported by Simons research grant, NSF DMS 150908, ERC Gemis, DMS-1265230, DMS-1201475 OISE-1242272 PASI. Simons collaborative Grant - HMS, and partially supported by the Laboratory of Mirror Symmetry NRU HSE, RF government grant, ag. 14.641.31.000.

\subsection{General conventions}

\begin{itemize}

\item Let $\U$ be a Grothendieck universe, with $\U$ satisfying the axiom of infinity. The $\U$-small mathematical objects will be called only \emph{small}. We assume the axiom of Universes. Some arguments in this article will require to enlarge the universe $\U$, which is always possible by assuming the axiom of Universes. If $\V$ is such an enlargement in which $\U$ is small, the $\V$-small mathematical objects will be called \emph{not necessarily small}. 

\item We work within the theory of $\infty$-categories in the sense of Lurie \cite{htt}, a.k.a quasicategories. We follow the terminology and the conventions of loc.cit. regarding the theory of $\oo$-categories. We denote by $\catinf$ the $\oo$-category of small $\oo$-categories.

\item For two $\oo$-categories $\ccal$ and $\ccal'$, we denote by $\fun(\ccal, \ccal')$ the $\oo$-category of functors from $\ccal$ to $\ccal'$. 

\item If $\ccal$ is an $\oo$-category admitting a final object $\ast$. The $\oo$-category $\ccal_*$ of pointed objects in $\ccal$ is by definition the full subcategory of $\fun(\del^1, \ccal)$ consisting of morphisms $C\mo C'$ such that $C$ is a final object of $\ccal$.

\item We denote by $\spaces$ the $\infty$-category of small spaces in the sense of \cite[Def 1.2.16.1]{htt}. It is therefore defined as the simplicial nerve of the simplicial category of small Kan complexes. We denote by $\bspaces$ the $\oo$-category of not necessarily small spaces. The word \emph{space} means an object of the $\oo$-category $\spaces$. The expression \emph{not necessarily small space} means an object of $\bspaces$ for some appropriate universe $\V$. We denote by $\ast$ the final object $\del^0$ of $\spaces$.

\item We follow the terminology of \cite{ha} concerning $n$-connective spaces. Let $n\geq 0$ be an integer. A space $X$ is \emph{$n$-connected} if $\pi_i(X,x)=0$ for every $i<n$ and every vertex $x\in X$. Every space is declared to be $(-1)$-connective. 

\item We follow the terminology of \cite{ha} concerning $n$-truncated spaces. Let $n\geq -1$ be an integer. A space $X$ is \emph{$n$-truncated} if $\pi_i(X,x)=0$ for every $i>n$ and every vertex $x\in X$. By convention a space is $(-1)$-truncated if it is either empty or contractible, and is $(-2)$-truncated if it is contractible. A map of pointed spaces $X\mo Y$ is called $n$-truncated if its fiber is $n$-truncated. 

\item We denote by $\Sp$ the stable $\oo$-category of spectra (see \cite[Def 1.4.3.1]{ha}).

\item Let $n\in \Z$ be an integer. We consider the usual t-structure on the stable $\oo$-category $\Sp$. A spectrum $X$ is $n$-connective (resp. $n$-truncated) if $\pi_i X=0$ for every $i<n$ (resp. $i>n$). 

\end{itemize}

\section{Generalities on formal moduli problems}

We start by recalling from \cite[§4]{dagX} the notion of formal moduli problems defined over $\E_n$-algebras and their classification. Then we recall the setting of axiomatic deformation theory from \cite[§1]{dagX}, in which our results from the next subsection §\ref{apploop} concerning the approximation of formal moduli via loop spaces are proven.

\subsection{Reminders on formal moduli problems}

\subsubsection*{Reminders on formal $\en$-moduli problems}\label{sssenfmp}

Let $k$ be a field of arbitrary characteristic. Let $\modk$ denote the $\oo$-category of $k$-module spectra. 
Let $n\geq 0$ be an integer and let $\algenk:=\alg_{\E_n} (\modk)$ denote the $\infty$-category of $\E_n$-algebras over $k$ in the sense of \cite[Def 7.1.3.5]{ha}. Let $\algenka=(\algenk)_{/k}$ denote the $\oo$-category of augmented $\en$-algebras. We call an $\en$-algebra $A$ over $k$ \emph{artinian} if the following conditions are satisfied 
\begin{itemize}
\item $\pi_i A=0$ for $i<0$ and for $i\gg 0$. 
\item For every $i$, the $k$-vector space $\pi_i A$ is finite dimensional. 
\item Let $\mathfrak{r}$ denote the radical of the algebra $\pi_0 A$, then the natural map $(\pi_0 A)/\mathfrak{r}\mo k$ is an isomorphism. 
\end{itemize} 
We denote by $\algenkart\subseteq \algenk$ the full subcategory spanned by artinian $\E_n$-algebras. For every integer $m\geq 0$, we denote by $k\oplus k[m]$ the shifted square zero extension viewed as $\en$-algebras over $k$, which are examples of artinian $\en$-algebras. 

A functor $X:\algenkart\mo \spaces$ is called a \emph{formal $\en$-moduli problem} if the space $X(k)$ is contractible and if for every pullback diagram 
$$\xymatrix{ A \ar[r]^-{ } \ar[d]^-{ } & A_0 \ar[d]^-{ } \\ A_1  \ar[r]^-{ } & A_{01}  }$$
in $\algenkart$ such that the maps $\pi_0 A_0 \mo \pi_0 A_{01} \leftarrow \pi_0 A_1$ are surjective, the induced square 
$$\xymatrix{ X(A) \ar[r]^-{ } \ar[d]^-{ } & X(A_0) \ar[d]^-{ } \\ X(A_1)  \ar[r]^-{ } & X(A_{01})  }$$
is a pullback diagram in $\spaces$. Let $\fun_*(\algenkart, \spaces)\subseteq \fun(\algenkart, \spaces)$ denote the full subcategory spanned by those functors $X:\algenkart\mo \spaces$ such that $X(k)$ is a contractible space. We denote by $\fmpen\subseteq \fun_*(\algenkart, \spaces)$ the full subcategory spanned by formal $\en$-moduli problems.  The class of pullback diagrams of the form
$$\xymatrix{ k\oplus k[m] \ar[r]^-{ } \ar[d]^-{ } & k \ar[d]^-{ } \\ k  \ar[r]^-{ } & k\oplus k[m+1]  }$$
for $m\geq 0$ are used to defined the tangent complex of a formal $\en$-moduli problem. Indeed if $X$ is a formal $\en$-moduli problem over $k$, the induced map $X(k\oplus k[m])\mo \Omega X(k\oplus k[m+1])$ is a homotopy equivalence, and defines a spectrum $T_X$ called the tangent complex of $X$ (see \cite[Def 1.2.5]{dagX}). By \cite[Prop 1.2.10]{dagX}, a map of formal $\en$-moduli problems $X\mo Y$ is an equivalence if and only if the induced map $T_X\mo T_Y$ is an equivalence of spectra. 

By \cite[Thm 4.0.8]{dagX} there exists a functor 
$$\Psi : \algenka\lmo \fmpen$$
which is an equivalence of $\oo$-categories. For an augmented $\en$-algebra $B$ and for an artinian $\en$-algebra $A$ over $k$, there is an equivalence $\Psi_B(A)\simeq \map_{\algenka} (\dn(A), B)$ where $\dn: (\algenka)^\op\mo \algenka$ is the Koszul duality functor defined in \cite[Prop 4.4.1]{dagX}. Moreover there exists a commutative diagram of $\infty$-categories 
$$\xymatrix{ \algenka \ar[r]^-{ \Psi} \ar[d]_-{\mfr } & \fmpen \ar[d]^-{ T[-n]} \\  \modk \ar[r]^-{ } & \Sp }$$
where $\mfr$ is given by taking the augmentation ideal.

\subsubsection*{Reminders on axiomatic deformation theory}

We recall some definitions and constructions from \cite[§1]{dagX}.

For $\A$ a presentable $\oo$-category, we denote by $\Sp(\A)$ the $\oo$-category of spectra objects of $\A$ in the sense of \cite[Def 1.4.2.8]{ha}. By \cite[Rem 1.4.2.25]{ha} there exists an equivalence of $\oo$-categories $\Sp(\A)\simeq \lim (\hdots\mo \A_*\mos{\Omega} \A_*\mos{\Omega} \A_*)$. We denote by $\Omega^{\oo} : \Sp(\A)\mo \A$ the evaluation on the $0$-sphere $S^0$ and for every integer $n\in \Z$, we denote by $\Omega^{\oo-n} : \Sp(\A)\mo \A$ the composition of $\Omega^{\oo}$ with the $n$th shift functor on $\Sp(\A)$. 

A \emph{deformation context} in the sense of \cite[Def 1.1.3]{dagX} is a pair $(\A, (E_\alpha)_{\alpha\in T})$ where $\A$ is a presentable $\infty$-category and $(E_\alpha)_{\alpha\in T}$ is a sequence of objects of the $\infty$-category $\Sp(\A)$ of spectra objects in $\A$. Given a deformation context $(\A, (E_\alpha)_{\alpha\in T})$, a map $f:A\mo A'$ in $\A$ is called \emph{elementary} if there exists an index $\alpha\in T$, an integer $n\geq 0$ and a pullback diagram
$$\xymatrix{  A \ar[r]^-{ } \ar[d]_-{f } &  \ast_\A \ar[d]^-{ } \\  A '\ar[r]^-{ } & E_\alpha[n] }$$
in $\A$, where $\ast_\A$ is a final object of $\A$. A map $f : A\mo A'$ in $\A$ is called \emph{small} if it can be written as a finite composition of elementary maps $A_0=A\mo A_1\mo \hdots \mo A_m=A'$ in $\A$. An object $A\in \A$ is called \emph{artinian} if the map $A\mo \ast_\A$ is small. We denote by $\Aart$ the full subcategory of $\A$ spanned by artinian objects.

Given a deformation context $(\A, (E_\alpha)_{\alpha\in T})$, recall that a \emph{formal moduli problem} in the sense of \cite[Def 1.1.14]{dagX} is a functor $X:\Aart\mo \spaces$ which satisfies the two conditions 
\begin{enumerate}
\item The space $X(\ast_\A)$ is contractible. 
\item For every pullback diagram
$$\xymatrix{ A \ar[r]^-{ } \ar[d]^-{ } & A_0 \ar[d] \\ A_1 \ar[r]^-{ } & A_{01} }$$
in $\Aart$ such that the map $A_0\mo A_{01}$ is small, the induced square 
$$\xymatrix{ X(A) \ar[r]^-{ } \ar[d]^-{ } & X(A_0) \ar[d] \\ X(A_1) \ar[r]^-{ } & X(A_{01}) }$$
is a pullback diagram in $\spaces$. 
\end{enumerate}

By \cite[Prop 1.1.15]{dagX} we are allowed to test condition (2) on the smaller class of pullback diagrams of the form 
$$\xymatrix{ A \ar[r]^-{ } \ar[d]^-{ } & \ast_\A \ar[d] \\ A' \ar[r]^-{ } & E_\alpha[n] }$$
in $\Aart$ where $\alpha\in T$ and $n\geq 0$. 

We denote by $\fun_*(\Aart, \spaces)\subseteq \fun(\Aart, \spaces)$ the subcategory spanned functors satisfying condition (1) above, or in other words sending $\ast_\A$ to a contractible space. We denote by $\fmpa\subseteq\fun_*(\Aart, \spaces)$ the subcategory spanned by the formal moduli problems. 

\begin{ex}\label{exdefconalgen}
Let $k$ be a field. Let $n\geq 1$ and $\A=\algenka$ be the presentable $\oo$-category of augmented $\en$-algebras over $k$. Consider the deformation context $(\A, \{k\})$ where $k$ is seen as a module over itself. The deloopings of $k$ are given by the shifted square zero extensions $k\oplus k[m]$ for $m\geq 0$. By \cite[Prop 4.5.1]{dagX}, an augmented $\en$-algebra $A$ is artinian with respect to the deformation context $(\A, \{k\})$ if and only if it is artinian in the sense of §\ref{sssenfmp}. By \cite[Prop 4.5.1]{dagX}, a morphism $A\mo A'$ in $\algenka$ is small with respect to the deformation context $(\A, \{k\})$ if and only if the induced map $\pi_0A\mo \pi_0 A'$ is surjective. Then \cite[Cor 4.5.4]{dagX} implies that $X$ is a formal moduli problem with respect to the deformation context $(\A, \{k\})$ if and only if it is a formal $\en$-moduli problem in the sense of §\ref{sssenfmp}.  
\end{ex}

\begin{nota}\label{notasf}
Let $(\A, (E_\alpha)_{\alpha\in T})$ be a deformation context. Consider the embedding $i_\A: \fmpa\hookrightarrow \fun_*(\Aart, \spaces)$ into the $\oo$-category of functors $X:\Aart \mo \spaces$ such that $X(\ast_\A)$ is a contractible space. The $\oo$-categories $\fmpa$ and $\fun_*(\Aart, \spaces)$ are presentable and by definition of formal moduli problems the functor $i_\A$ preserves all small limits. By the adjoint functor theorem \cite[Cor 5.5.2.9]{htt}, this implies that $i_\A$ admits a left adjoint denoted by 
$$\xymatrix{\fmpa \ar@{^{(}->}[r]_{i_\A}  &  \fun_*(\Aart, \spaces) \ar@/_1pc/[l]_{L_\A}}.$$

For every functor $X:\Aart\mo \spaces$ sending $\ast_\A$ to a contractible space we have the unit natural transformation $X\mo i_\A(L_\A(X))$ of this adjunction. The formal moduli problem $L_\A(X)$ is therefore initial among formal moduli problems $Y$ having a natural transformation $X\mo i_\A Y$. 
\end{nota}

Let $(\A, (E_\alpha)_{\alpha\in T})$ be a deformation context. Recall from \cite[Def 1.2.5]{dagX} that every formal moduli problem $X:\Aart\mo \spaces$ has a tangent spectrum $X(E_\alpha)\in \Sp$ of $X$ at $\alpha$. Moreover a map $X\mo Y$ of formal moduli problems is an equivalence if and only if for every $\alpha$ the map $X(E_\alpha)\mo Y(E_\alpha)$ is an equivalence of spectra. 

In \cite[Def 1.3.1 and 1.3.9]{dagX}, Lurie defines the notion of \emph{weak deformation theory} and of \emph{deformation theory}. In short a weak deformation theory for $(\A, (E_\alpha)_{\alpha\in T})$ consists of a functor $\deth : \A^{op}\mo \B$ such that $\B$ is a presentable $\oo$-category, $\deth$ has a left adjoint $\deth'$ which behaves like an inverse equivalence to $\deth$ on a certain subcategory of $\B$. In this situation, \cite[Cor 1.3.6]{dagX} asserts the existence of an $\oo$-functor 
$$\Psi : \B\mo \fmpa$$ 
which is given on objects by $\Psi_B(A)=\map_{\B} (\deth(A), B)$. By applying the tangent spectrum we obtain for every $\alpha$ a functor $e_\alpha : \B\mo \Sp$ given on objects by $e_\alpha(B)=\Psi_B(E_\alpha)$. A deformation theory for $(\A, (E_\alpha)_{\alpha\in T})$ is a weak deformation theory such that each $ e_\alpha$ preserves small sifted colimits and such that a morphism $f$ in $\B$ is an equivalence if and only if each $e_\alpha(f)$ is an equivalence of spectra. The main theorem \cite[Thm 1.3.12]{dagX} then asserts that under the existence of a deformation theory for $(\A, (E_\alpha)_{\alpha\in T})$, the functor $\Psi$ is an equivalence of $\oo$-categories.


\subsubsection*{Formal spectrum and Koszul duality}

We now set up some notations and intermediate results concerning the formal spectrum, pro-objects and formal Koszul duality.

\begin{nota}\label{notaspf}
Let $(\A, (E_\alpha)_{\alpha\in T})$ be a deformation context and let $A\in \A$ be any object. Consider the functor $\spf(A): \Aart\mo \spaces$ which is co-represented by $A$ given by the formula $\spf(A)(R)=\map_\A(A,R)$. Then $\spf(A)$ is a formal moduli problem which is called the\emph{ formal spectrum of $A$}. Moreover the construction $A\mapsto \spf(A)$ determines a functor 
$$\spf:\A^\op\lmo \fmpa$$
which is fully faithful when restricted to the subcategory $(\Aart)^\op$ spanned by artinian objects. 
\end{nota} 

\begin{ex}
Let $n\geq 1$ and consider the deformation context $(\algenka, k)$ formed by augmented $\en$-algebras over $k$. Let $\epsilon: A\mo k$ be an augmented $\en$-algebra over $k$. If we think about the $\en$-algebra $A$ as determining a noncommutative scheme $\spec(A)$, then the formal spectrum $\spf(A)$ can be thought of as the functor parametrizing deformations of the point of $\spec(A)$ determined by $\epsilon$ over artinian $\en$-algebras. 
\end{ex}

The following Lemma will be useful in the sequel and follows from the fact that the Koszul duality functor $\dn$ is a deformation theory. 

\begin{lem}\label{lemalgrepspf}
Let $n\geq 1$ be an integer and denote by $\tfr^{(n)}: \fmpen\mo \algenka$ the equivalence given by \cite[Thm 4.0.8]{dagX}. Let $B$ be an augmented $\E_n$-algebra over $k$. Then there exists a natural equivalence of augmented $\E_n$-algebras $\tfr^{(n)}_{\spf(B)}\simeq \dn(B)$ or equivalently an equivalence of formal moduli problems $\spf(B)\simeq \Psi_{\dn(B)}$. 
\end{lem} 

\begin{proof} 
Consider the natural transformation $\alpha: \spf(B)\mo \Psi_{\dn(B)}$ induced by the natural map 
$$\map_{\algenka} (B,A)\lmos{\alpha_A}  \map_{\algenka} (\dn(A),\dn(B))$$
for each artinian $\E_n$-algebra $A$ over $k$. By adjunction the target space of $\alpha_A$ is naturally equivalent to the space $\map_{\algenka} (B,\dn\dn(A))$ and the map is induced by the unit map $A\mo \dn\dn(A)$. Let $m\geq 0$ be an integer and $A=k\oplus k[m]$ be the trivial square zero extension. We claim that the biduality map $A\mo \dn\dn(A)$ is an equivalence. Indeed by \cite[Prop 4.5.6]{dagX} there exists an equivalence $\dn (\freea_{\E_n} (k[-m-n]) \simeq k\oplus k[m]$ of augmented $\E_n$-algebras and by \cite[Prop 4.1.13]{dagX} the $\E_n$-algebra $\freea_{\E_n} (k[-m-n])$ is $n$-coconnective and locally finite whenever $m\geq 0$, which by \cite[Thm 4.4.5]{dagX} implies that the biduality map is an equivalence for the $\E_n$-algebra $\freea_{\E_n} (k[-m-n])$, and thus also for $A=k\oplus k[m]$. We deduce that the map $\alpha_A$ is a homotopy equivalence and therefore that $\alpha$ induces an equivalence on the tangent spectrum, and is therefore an equivalence, which proves the lemma. 
\end{proof}


\subsection{Approximation of formal moduli and loop spaces}\label{apploop}

As explained above, the deformation functor of a linear $\oo$-category over a field does not in general satisfy the derived Schlessinger conditions, but turns out to be rather close to satisfy them. Many examples of deformation functors of some algebro-geometric objects behave similarly. In \cite[§5.1]{dagX}, Lurie developed a formalism to study these functors under the name of $m$-proximate formal moduli problems for $m\geq 0$ an integer.  We will prove that the functor $L=L_\A:\fun_*(\Aart, \spaces)\mo \fmpa$ from Notation \ref{notasf} admits a special description when restricted to the subcategory spanned by $m$-proximate formal moduli problems in terms of $\E_m$-maps between certain loop $\E_m$-spaces (see Corollary \ref{corassfmp}). For this we will prove a result concerning the loop space functor on the $\oo$-category of formal moduli problem (see Proposition \ref{propfmpgp}).

Let $(\A,(E_\alpha)_{\alpha\in T})$ be a deformation context and $m\geq 0$ an integer. Recall from \cite[Def 5.1.5]{dagX} that a functor $X:\Aart\mo \spaces$ is called an \emph{$m$-proximate formal moduli problem} if the space $X(\ast_\A)$ is contractible and if for every pullback square
$$\xymatrix{ A \ar[r]^-{ } \ar[d]^-{ } & A_0 \ar[d] \\ A_1 \ar[r]^-{ } & A_{01} }$$
in $\Aart$ such that the map $A_0\mo A_{01}$ is small, the induced map
$$X(A)\mo X(A_0)\times_{X(A_{01})} X(A_1)$$
is $(m-2)$-truncated (that is, all homotopy fibers of this map are $(m-2)$-truncated). We observe that $0$-proximate formal moduli problems are exactly the formal moduli problems. 

\begin{nota}
We denote by $\proxma\subseteq \fun_*(\Aart, \spaces)$ the full subcategory spanned by $m$-proximate formal moduli problems. There is a tower of $\infty$-categories
$$\fmpa=\prox(0)_\A\subseteq \prox(1)_\A \subseteq \prox(2)_\A  \subseteq \hdots\subseteq \fun_*(\Aart, \spaces)$$
For example, we will see below that the deformation functor of an object in a linear $\infty$-category is in general 1-proximate and that the deformation functor of a fixed linear $\infty$-category itself is in general 2-proximate. 
\end{nota}

\begin{nota}
Let $n\geq 1$ and $m\geq 0$, and consider the deformation context $(\algenka, k)$ of  Example \ref{exdefconalgen}. We will denote by $\prox(m)_{\en}\subseteq \fun_*(\algenkart, \spaces)$ the full subcategory of $m$-proximate formal $\en$-moduli problem. 
\end{nota}

\begin{nota}\label{notaloopprefmp}
Let $(\A,(E_\alpha)_{\alpha\in T})$ be a deformation context and suppose that $\A$ is a pointed $\oo$-category in the sense that any final object is also initial (e.g. the $\oo$-category $\algenka$ of augmented $\en$-algebras over $k$). Then the $\oo$-category $\fun_*(\Aart, \spaces)$ is also pointed. We denote by  
$$\Omega : \fun_*(\Aart, \spaces)\lmo \fun_*(\Aart, \spaces)$$ 
the loop functor in the sense of \cite[Rem 1.1.2.9]{ha}. The embedding $\fmpa\hookrightarrow \fun_*(\Aart, \spaces)$ 
commutes with small limits and the loop functor restricts to a functor $\Omega:\fmpa\mo \fmpa$. For $m\geq 0$ an integer, we denote by $\Omega^m = \Omega\circ \hdots\circ \Omega$ ($m$ times) the iterated loop functor. 
\end{nota}

\begin{rema} 
In the situation of Notation \ref{notaloopprefmp}, let $X:\Aart\mo \spaces$ be any functor. Then the functor $\Omega X$ is given on objects by $(\Omega X)(A)=\Omega X(A)$ where the loop space is taken at the base point $\ast\mo X(A)$ corresponding to the map $\ast \simeq X(\ast_\A)\mo X(A)$ induced by the essentially unique map $\ast_\A\mo A$. 

If we imagine that the functor $X$ parametrizes deformations of an object of an $\oo$-category corresponding to the connected component $\pi_0 X(\ast_\A)=\ast$, then the functor $\Omega X$ parametrizes deformations of the identity of this object, or in other words infinitesimal autoequivalences of this object. 
\end{rema}

\begin{rema}\label{remaproxom}
Let $f:X\mo Y$ be a Kan fibration between Kan complexes. Let $m\geq 0$ be an integer and suppose that the map $f$ is $(m-2)$-truncated, that is, all its homotopy fibers are $(m-2)$-truncated. Then we deduce from the long exact sequence of homotopy groups for any choice of basepoint that the induced map $\Omega^m X\mo \Omega^m Y$ is a homotopy equivalence, and whenever $m\geq 1$ that the morphism $\pi_{m-1}(X,x_0) \mo \pi_{m-1} (Y,f(x_0))$ is injective for any basepoint $x_0\in X$. Besides, we deduce from the same long exact sequence that the induced map $\Omega X\mo \Omega Y$ is $(m-3)$-truncated. 
\end{rema}

\begin{lem}\label{lemproxfmp}
Let $(\A,(E_\alpha)_{\alpha\in T})$ be a deformation context such that $\A$ is a pointed $\oo$-category. Let $m\geq 0$ and let $X:\Aart\mo \spaces$ be an $m$-proximate formal moduli problem. Then the natural map $\Omega^m X\mo \Omega^m LX$ induced by the unit map $X\mo LX$ is an equivalence. In particular the functor $\Omega^m X$ is a formal moduli problem. 
\end{lem} 

\begin{proof} 
By \cite[Thm 5.1.9]{dagX} the functor $X$ is $m$-proximate if and only if the unit map $X\mo LX$ is $(m-2)$-truncated. If it is the case, we deduce from Remark \ref{remaproxom} that the induced map $\Omega^m X\mo \Omega^m Y$ is an equivalence.
\end{proof}

\begin{nota}\label{notagpmon}
Let $\xcal$ be an $\oo$-category which admits finite products. Consider the cartesian symmetric monoidal structure $\xcal^\times$ with monoidal product given by the cartesian product (see \cite[§2.4.1]{ha}). Let $0\leq m\leq \oo$ and consider the $\oo$-category $\monem(\xcal)=\alg_{\E_m} (\xcal^\times)$ of $\E_m$-monoid objects in $\xcal$. Recall that an $\E_1$-monoid object $M$ in $\xcal$ with multiplication map $m:M\times M\mo M$ is called grouplike if the two maps 
$$(m,p_1):M\times M\mo M\times M, \quad (m,p_2):M\times M\mo M\times M$$
are equivalences in $\xcal$, where $p_i$ is the $i$th projection map. Recall that an $\E_m$-monoid object $M$ in $\xcal$ is called grouplike if the underlying $\eo$-monoid (for any embedding of $\oo$-operads $\eo^\te \hookrightarrow \E_m^\te$) is grouplike. Let $\monemgp(\xcal)\subseteq \monem(\xcal)$ denote the full subcategory spanned by grouplike $\E_m$-monoid objects (see \cite[Def 5.2.6.6]{ha}). If $\xcal$ is a presentable $\oo$-category, then an $\E_m$-monoid object $M$ in $\xcal$ is grouplike if and only if the monoid $\pi_0M$ is a group. 

If $\xcal$ is a pointed presentable $\oo$-category, by \cite[Not 5.2.6.11]{ha} the iterated loop functor $\Omega^m : \xcal\mo \xcal$ factorizes as a functor
$$\Omega^m:\xcal\mo \monemgp(\xcal),$$
where by convention, if $m=\oo$ we set $\Omega^\oo=\lim_m \Omega^m$ via the equivalence $\lim_m \monem(\xcal)\simeq \monei(\xcal)$. 
\end{nota} 

\begin{rema}\label{remacomgp}
Let $\xcal$ be an $\oo$-category which admits finite products. The property of being grouplike for an $\E_m$-monoid object is independent of the choice of the embedding of $\oo$-operads $\eo^\te \hookrightarrow \E_m^\te$ (see \cite[Rem 5.2.6.8]{ha}). This implies that the equivalence $\lim_m \monem(\xcal)\simeq \monei(\xcal)$ restricts to an equivalence $\lim_m \monemgp(\xcal)\simeq \moneigp(\xcal)$. 
\end{rema}

\begin{rema}\label{remaaddgp}
Let $\xcal$ be an $\oo$-category which admits finite products. Consider the full embedding $\monemgp(\xcal)\subseteq \monem(\xcal)$ of grouplike $\E_m$-monoid objects into all $\E_m$-monoid objects. This inclusion preserves finite product and therefore extends to a symmetric monoidal functor. We have an induced functor $\mon_{\eo}(\monemgp(\xcal))\mo \mon_{\eo}(\monem(\xcal))$ on the $\oo$-categories of $\eo$-monoid objects which is fully faithful. By the additivity theorem \cite[Thm 5.1.2.2]{ha} the target of this functor is equivalent to the $\oo$-category $\mon_{\E_{m+1}}(\xcal)$ of $\E_{m+1}$-monoid objects and by definition of grouplike objects the equivalence $\mon_{\eo}(\monem(\xcal))\simeq \mon_{\E_{m+1}}(\xcal)$ restricts to an equivalence $\mongp_{\eo}(\monem(\xcal))\simeq \mon_{\E_{m+1}}^{\mrm{gp}}(\xcal)$. Moreover the above embedding restricts to an embedding 
$$\mongp_{\eo}(\monemgp(\xcal))\mo \mongp_{\eo}(\monem(\xcal))\simeq \mon_{\E_{m+1}}^{\mrm{gp}}(\xcal)$$ on the subcategories of grouplike $\eo$-monoid objects. This embedding is an equivalence of $\oo$-categories. Indeed, it is also essentially surjective, for if $M$ is a grouplike $\E_1$-object in $\monem(\xcal)$, it follows from Remark \ref{remacomgp} that the underlying $\E_m$-monoid object of $M$ is grouplike and lies in the essential image.
\end{rema}

The following result provides conditions on the underlying deformation context under which the iterated loop functor is an equivalence for the $\oo$-category $\fmpa$ of formal moduli problems. 

\begin{prop}\label{propfmpgp}
Let $(\A, (E_\alpha)_{\alpha\in T})$ be a deformation context admitting a deformation theory and such that $\A$ is a pointed $\oo$-category. Then for every $0\leq m\leq \oo$, the iterated loop functor $\Omega^m:\fmpa\mo \monemgp(\fmpa)$ is an equivalence of $\infty$-categories. 
\end{prop} 

Proposition \ref{propfmpgp} is a consequence of the following Lemma. 

\begin{lem}\label{lemgpeq} Let $\xcal$ be a pointed presentable $\oo$-category. Suppose that there exists a presentable stable $\oo$-category $\acal$ and a functor $f:\xcal\mo \acal$ which is conservative, preserves geometric realizations of simplicial objects and preserves finite limits. Then for any $0\leq m\leq \oo$ the functor $\Omega^m : \xcal\mo \monemgp(\xcal)$ is an equivalence of $\oo$-categories. 
\end{lem} 

\begin{proof} If $m=0$ the assertion follows from the assumption that $\xcal$ is pointed, since there are equivalences $\mon_{\E_0}^{\mrm{gp}}(\xcal)\simeq \mon_{\E_0}(\xcal)\simeq \xcal_*$, where $\xcal_*$ is the $\oo$-category of pointed objects. Otherwise for finite $m$, by Remark \ref{remaaddgp} and ascending induction on $m$ it suffices to prove the result for $m=1$. For this, consider the diagram 
$$\xymatrix{ \xcal  \ar[r]^-{ \Omega_\xcal } \ar[d]_-{ f} & \mongp_{\E_1} (\xcal) \ar[d]^-{f^{\mrm{gp}} } \\  \acal \ar[r]^-{\Omega_\acal } & \mongp_{\E_1} (\acal)  .}$$
Because $f$ preserves finite limits, there exists a natural equivalence $f^{\mrm{gp}}\Omega_\xcal\simeq \Omega_\acal  f$ making this diagram commutative up to equivalence. The $\oo$-category $\acal$ being stable, the functor $\Omega_\acal$ is an equivalence . The functors $\Omega_\xcal$ and $\Omega_\acal$ admit as left adjoint the functors $B_{\xcal}$ and $B_{\acal}$ respectively given by the bar construction. By virtue of the assumption that $f$ preserves geometric realizations of simplicial objects, and that the bar construction is given by the geometric realization of a simplicial object (see \cite[Rem 5.2.2.8]{ha}), there exists a natural equivalence $f B_\xcal\simeq B_\acal  f^{\mrm{gp}}$. Let $\epsilon: B_\xcal \Omega_\xcal\mo \id_\xcal$ be the counit map. The composition of $\epsilon$ with the composite functor $\Omega_\acal f$ is equivalent to the natural transformation 
$$ f^{\mrm{gp}} \Omega_\xcal \simeq \Omega_\acal B_\acal  f^{\mrm{gp}} \Omega_\xcal  \simeq \Omega_\acal f B_\xcal \Omega_\xcal \mo \Omega_\acal f$$
which is an equivalence because $f$ preserves finite limits. Because $\Omega_\acal$ is an equivalence and $f$ is conservative, this implies that $\epsilon$ is an equivalence. We can prove similarly that the unit map $\eta: \id\mo \Omega_\xcal B_\xcal$ is an equivalence by composing with the functor $f^{\mrm{gp}}$. 
The assertion for $m=\oo$ follows by passing to the limit over $m$ and form Remark \ref{remacomgp}.  
\end{proof}

\begin{proof}[Proof of Proposition \ref{propfmpgp}] Let $\deth:\A^{op} \mo \B$ be a deformation theory for the deformation context $(\A, (E_\alpha)_{\alpha\in T})$. By definition, we have at our disposal for every $\alpha\in T$, a functor $e_\alpha : \B\mo \Sp$ which preserves small sifted colimits and such that the family $\{e_\alpha\}_\alpha$ is jointly conservative (see \cite[Cor 1.3.8, Def 1.3.9]{dagX}). This implies that the functor 
$$e=\prod_{\alpha\in T} e_\alpha : \B\mo \prod_{\alpha\in T}\dal_\alpha=\acal$$ 
where $\acal_\alpha=\Sp$ for every $\alpha$, is conservative and preserves small sifted colimits. By construction, the functor $e_\alpha$ preserves finite limits, so that $e$ has the same property. Denote by $\tfr : \fmpa \mo \B$ the inverse of the equivalence $\Psi:\B\mo \fmpa$ (see \cite[Thm 1.3.12]{dagX}). Then the composite functor $e\circ \tfr : \fmpa\mo \acal$ satisfies all the assumptions of Lemma \ref{lemgpeq}, by which we deduce the claim. 
\end{proof}

\begin{nota}
In the situation of Proposition \ref{propfmpgp}, we denote by $B^m:\monemgp(\fmpa)\mo \fmpa$ the inverse of the equivalence $\Omega^m:\fmpa\mo \monemgp(\fmpa)$. 
\end{nota}

\begin{nota}
Let $(\A, (E_\alpha)_{\alpha\in T})$ be a deformation context such that $\A$ is a pointed $\oo$-category. Consider the restriction $\Omega^m_{\mrm{prox}}: \prox(m)_\A\mo \monemgp(\fun_*(\Aart, \spaces))$ of the iterated loop functor $\Omega^m:\fun_*(\Aart, \spaces)\mo \monemgp(\fun_*(\Aart, \spaces))$. By Lemma \ref{lemproxfmp}, the functor $\Omega^m_{\mrm{prox}}$ factorizes as a functor $\Omega^m_{\mrm{prox}}: \prox(m)_\A\mo \monemgp(\fmpa)$. 
\end{nota}

\begin{cor}\label{corassfmp}
Let $(\A, (E_\alpha)_{\alpha\in T})$ be a deformation context admitting a deformation theory and such that $\A$ is a pointed $\oo$-category. Let $m\geq 0$ and let $L_m: \prox(m)_\A\mo \fmpa$ denote the restriction of the functor $L$ to the subcategory $\prox(m)_\A$ of $m$-proximate formal moduli problems. Then there exists an equivalence of functors $L_m\simeq B^m\Omega^m_{\mrm{prox}}: \prox(m)_\A\mo \fmpa$. In particular, the functor $L_m$ preserves all small limits. As a consequence, for every $m$-proximate formal moduli problem $X$, the functor $L_mX$ is given by the formula
$$L_m X(A)\simeq \map_{\monemgp(\fmpa)} (\Omega^m \spf(A), \Omega^m X).$$
\end{cor}

\begin{proof} 
By Lemma \ref{lemproxfmp}, there is an equivalence of functors $\Omega^m_{\mrm{prox}} \simeq \Omega^m L_m$, which after composition with $B^m$ provides by Proposition \ref{propfmpgp} the expected natural equivalence $B^m \Omega^m_{\mrm{prox}}\simeq L_m$. Because $\Omega^m_{\mrm{prox}}$ is given by a limit and that $B^m$ is an equivalence, we have that $L_m$ preserves all small limits. Let $X$ be an $m$-proximate formal moduli problem. Then for $A\in \Aart$ we have natural homotopy equivalences
\begin{align*}
L_mX(A)\simeq B^m\Omega^m_{\mrm{prox}} X (A) & \simeq \map_{\fmpa}(\spf(A), B^m\Omega^m X)\\
 & \simeq \map_{\monemgp(\fmpa)} (\Omega^m \spf(A), \Omega^m X).
\end{align*} 
\end{proof}


\section{Reminders on deformations of categories}

We start with notations and terminology related to presentable, compactly generated and linear $\infty$-categories, as well as the relation with differential graded categories. We follow the terminology of \cite{htt} and \cite{ha}. We also point out the reference \cite{marcoth} where we can find an overview on this subject.

\subsection{Linear $\oo$-categories}\label{sslincat}

We denote by $\prl$ the $\infty$-category of not necessarily small presentable $\infty$-categories together with functors which are left adjoints (see \cite[Def 5.5.3.1]{htt} for a construction of $\prl$). 

An $\infty$-category $\ccal$ is said to be \emph{compactly generated} if it is presentable and if in addition there exists a small $\infty$-category $\ccal_0$ which admits finite colimits and an equivalence $\ind(\ccal_0)\simeq \ccal$. In this situation, by \cite[Thm 5.4.2.2]{htt} the small $\infty$-category $\ccal_0$ which generates $\ccal$ can be choosen to be the small idempotent complete full subcategory $\ccal^c\subseteq \ccal$ consisting of compact objects of $\ccal$.

By \cite[Prop 4.8.1.14]{ha}, the $\infty$-category $\prl$ admits a symmetric monoidal structure in the sense of \cite[Def 2.1.2.13]{ha} and we denote the induced symmetric monoidal product by $\te$. If $\mcal{C}, \mcal{C}'$ are presentable $\infty$-categories of the form $\mcal{C}=\pre(\mcal{C}_0)$ and $\mcal{C}'=\pre(\mcal{C}'_0)$ with $\ccal_0$, $\ccal_0$ small $\infty$-categories, then we have $\ccal\te \ccal=\pre(\ccal_0\times \ccal_0)$
where $\times$ is the cartesian symmetric monoidal structure on $\catinf$. The unit of $\prlte$ is the $\infty$-category $\spaces$ of small spaces. Moreover this symmetric monoidal structure is closed and for two presentable $\infty$-categories $\ccal$, $\ccal$, the $\infty$-category $\funl(\ccal, \ccal)$ of colimit preserving $\infty$-functors is presentable. By \cite[Lem 5.3.2.11]{ha}, if $\ccal$ and $\ccal$ are compactly generated $k$-linear $\infty$-categories, the tensor product $\ccal\te\ccal$ is again compactly generated and the subcategory $(\ccal\te \ccal)^c$ is generated under finite colimits by $\ccal^c\times \ccal^c$. 

Let $k$ be an $\ei$-ring. The $\infty$-category $\modk$ of $k$-module spectra admits a symmetric monoidal structure $\modk^\te$ given by the smash product over $k$. Considering $\modk^\te$ as an $\ei$-algebra object in the symmetric monoidal $\infty$-category $\prlte$, we define the $\infty$-category of $k$-linear $\infty$-categories to be
$$\prlk:=\Mod_{\modk^\te}(\prlte).$$
Morphisms in $\prlk$ are called $k$-linear functors. Because $\modk$ is a stable $\oo$-category in the sense of \cite[Def 1.1.1.9]{ha}, the underlying $\oo$-category of any $k$-linear $\oo$-category is stable. For example if $k=\s$ the sphere spectrum, then $\prl_\s$ is the $\oo$-category of presentable stable $\oo$-categories with exact functors between them. The symmetric monoidal structure on $\prl$ induces a symmetric monoidal structure on $\prlk$. We denote the corresponding tensor product by $\te_k$. Moreover the symmetric monoidal structure on $\catk$ is closed. For $k$-linear $\infty$-categories $\ccal$, $\ccalp$ and $\ccalpp$ there exists a $k$-linear $\infty$-category $\funl_k(\ccal, \ccal)$ of $k$-linear colimit preserving $\infty$-functors, such that there exists an equivalence 
$$\map_{\catk}(\ccal\te_k\ccalp, \ccalpp)\simeq \map_{\catk}(\ccal, \fun_k^L(\ccalp, \ccalpp)).$$
In the case $\ccal=\ccalp$, we denote by $\End_k(\ccal)=\funl_k(\ccal,\ccal)$ the $k$-linear $\oo$-category of $k$-linear endofunctors of $\ccal$. Any $k$-linear $\infty$-category $\ccal\in \catk$ is enriched over $\modk$. Indeed if $E,F\in \ccal$ are objects, the $\oo$-functor $\modk\mo \spaces$ given on objects by $V\mapsto \map_\ccal (V\te E,F)$ commutes with small colimits and by presentability of $\ccal$ is therefore representable by an object $\umap_\ccal (E,F)\in \modk$ satisfying the universal property 
$$\map_\ccal (V\te E,F)\simeq\map_{\modk}(M, \umap_\ccal (E,F)).$$
We set the notation $\extg^m_\ccal(E,F)=\pi_{-m} \umap_\ccal (E,F)$ for the Ext groups in $\ccal$. 

We now recall the relation of the above with the homotopy theory of $k$-linear differential graded categories. Let $k$ be a discrete commutative ring. We consider two $\infty$-categories of $k$-linear dg-categories (we use the terminology of \cite[§6.1.1]{marcoth}):

\begin{itemize}

\item $\dgklp$ : the $\infty$-category encoding the homotopy theory of presentable $k$-linear dg-categories up to quasi-equivalence. Morphisms are dg-functors which commute with small sums. The dg-nerve construction provides an $\infty$-functor
$$\nerdglp : \dgklp\lmo \prl.$$
 
\item $\dgkc$ : the $\infty$-category encoding the homotopy theory of presentable compactly generated $k$-linear dg-categories up to quasi-equivalence. Morphisms are dg-functors which commute with small sums and preserve compact objects. The dg-nerve restricts to an $\infty$-functor
$$\nerdgc : \dgkc\lmo \prlc.$$
\end{itemize}

By the work of Cohn \cite{cohndgst} there exists a commutative diagram of $\infty$-categories
$$\xymatrix{  \dgklp\ar[r]^-{\nerdglp }  & \catk  \\ \dgkc \ar[r]^-{ \nerdgc} \ar@{^{(}->}[u]^-{ } & \prlkc \ar@{^{(}->}[u]^-{ }} $$
where $\prlkc=\Mod_{\modk^\te} (\prlc)$ is the $\oo$-category of $k$-linear compactly generated $\oo$-categories together with $k$-linear functors which preserve small colimits. Moreover the bottom horizontal arrow is an equivalence of $\infty$-categories.

\begin{nota}
Let $k$ be an $\ei$-ring and $\ccal$ a $k$-linear $\infty$-category. If $\ecal$ is a collection of objects of $\ccal$, we denote by $\ecal^\perp$ the full subcategory of $\ccal$ consisting of objects $C$ such that $\umap_\ccal(E,C)\simeq 0$ for every $E\in \ecal$. 
\end{nota}

\begin{df} Let $k$ be an $\ei$-ring. Let $\ccal$ be a $k$-linear $\infty$-category. 
\begin{enumerate}
 
\item A collection $\ecal$ of objects of $\ccal$ \emph{generates $\ccal$} if $\ecal^\perp$ consists of zero objects of $\ccal$. 
\item A collection $\ecal$ of objects of $\ccal$ is \emph{a family of compact generators of $\ccal$} if each object $E\in \ecal$ is compact and if the collection $\ecal$ generates $\ccal$. 
\item A \emph{compact generator} of $\ccal$ is a compact object $E$ of $\ccal$ such that the family $\ecal=\{E\}$ generates $\ccal$. 
\end{enumerate}
\end{df}

\begin{df}\label{deftamely} Let $k$ be an $\ei$-ring. A $k$-linear $\oo$-category $\ccal$ is said to be \emph{tamely compactly generated} if $\ccal$ is compactly generated and if for any compact objects $E,E'$ of $\ccal$ we have $\extg^m_\ccal(E,E')=0$ for $m\gg 0$. 
\end{df}

\begin{rema}
Let $k$ be an $\ei$-ring and $\ccal$ a $k$-linear $\oo$-category. Suppose that $\ccal$ admits a single compact generator $E$. Then $\ccal$ is tamely compactly generated if and only if $\extg^m_\ccal (E,E)=0$ for $m\gg 0$. 
\end{rema}


\subsection{Deformations of objects}\label{ssobjdef}

We recall from \cite[§5.2]{dagX} the fundamental results about the deformation functor of an object in a linear $\oo$-category over a field. 

Let $k$ be a field, let $\ccal$ be a $k$-linear $\oo$-category and let $E\in \ccal$ be an object. Recall from \cite[Const 5.2.1]{dagX} that we can define a functor 
$$\objdef_E: \algkart\mo \spaces$$
from artinian $\eo$-algebras over $k$ to spaces which satisfies the following properties. 

\begin{itemize}
\item For every artinian $\eo$-algebra $A$ over $k$, the space $\objdef_E(A)$ classifies deformations of $E$ over $A$ in $\ccal$, or in other words it is the classifying space of pairs $(E_A, u)$ with $E_A\in \rmod_A(\ccal)$ a right $A$-module object in $\ccal$ and $u:E_A\te_A k\simeq E$ an equivalence in $\ccal$. 

\item In general, the functor $\objdef_E$ does not satisfy the derived Schlessinger conditions and is not a formal $\eo$-moduli problem. Let $\objdefa_E=L(\objdef_E)$ denote the associated formal $\eo$-moduli problem. 

\item By \cite[Cor 5.2.5]{dagX}, the functor $\objdef_E$ is a $1$-proximate formal $\eo$-moduli problem, or equivalently the unit map $\objdef_E\mo \objdefa_E$ is $(-1)$-truncated. 

\item Let $\uend_\ccal (E)$ be the $\eo$-algebra of endomorphisms of $E$ in $\ccal$. By \cite[Thm 5.2.8]{dagX} the augmented $\eo$-algebra which controls the formal $\eo$-moduli problem $\objdefa_E$ is given by $k\oplus \uend_\ccal (E)$. In other words there exists for every $A\in \algkart$ a natural equivalence 
$$\objdefa_E(A)\simeq \map_{\algk} (\du(A), \uend_\ccal (E)).$$
In particular the tangent spectrum of $\objdefa_E$ is given by $T_{\objdefa_E}\simeq \uend_\ccal (E)[1]$. 

\item Koszul duality for modules implies that for every artinian $\eo$-algebra $A$ we have a natural homotopy equivalence
$$\objdefa_E(A)\simeq \lmod_A^! (\ccal)^\simeq \times_{\ccal^\simeq} \{E\}$$
with the space of deformations of $E$ as an Ind-coherent right $A$-modules in $\ccal$, where $\lmod_A^! (\ccal)=\lmod_A^!\te_k \ccal$ is the $\oo$-category of Ind-coherent right $A$-modules in $\ccal$ (see \cite[Rem 5.2.16]{dagX}). 

\item By \cite[Prop 5.2.14]{dagX}, under the assumption that $\ccal$ admits a left complete $t$-structure and that $E$ is connective, then $\objdef_E$ is a formal $\eo$-moduli problem. 

\end{itemize}

\begin{nota}
Let $k$ be a field, let $\ccal$ be a $k$-linear $\oo$-category and let $E\in \ccal$ be an object. We will denote by $\objdefet_E$ the restriction of the functor $\objdef_E$ to artinian $\et$-algebras along the forgetful functor $\algetkart\mo \algkart$. 
\end{nota}


\subsection{Deformations of categories}\label{sscatdef}

We first recall from \cite[§5.3]{dagX} the fundamental results about the deformation functor of a linear $\oo$-category over a field. Then we recall the construction of this functor.

Let $k$ be field and $\ccal$ a $k$-linear $\oo$-category. In \cite[§5.3]{dagX} is constructed a functor 
$$\catdef_\ccal : \algetkart\lmo \bspaces$$
from artinian $\et$-algebras over $k$ to not necessarily small spaces which satisfies the following properties. 

\begin{itemize}
\item For any artinian $\et$-algebra $A$ over $k$, the space $\catdef_\ccal(A)$ classifies deformations of $\ccal$ over $A$, or in other words pairs $(\ccal_A, u)$ with $\ccal_A$ a right $A$-linear $\oo$-category and $u:\ccal_A\te_A k\simeq \ccal$ an equivalence in $\prlk$ (see below for the definition of linear $\oo$-category over an $\et$-algebra). We observe that the space $\catdef_\ccal(k)$ is contractible. 

\item In general, the functor $\catdef_\ccal$ does not satisfy the derived Schlessinger conditions, and is not an $\et$-formal moduli problem. Let $\catdefa_\ccal:=L(\catdef_\ccal)$ denote the associated $\et$-formal moduli problem and 
$$\sigma_\ccal : \catdef_\ccal\mo \catdefa_\ccal$$
denote the unit map. In other words, the map $\sigma_\ccal$ is in general not an equivalence. We will see a simple example below. 

\item However, the map $\sigma_\ccal$ is rather close to be an equivalence as shown by Lurie. Namely, by \cite[Cor 5.3.8]{dagX} the functor $\catdef_\ccal$ is a $2$-proximate $\et$-formal moduli problem in the sense of \cite[Def 5.1.5]{dagX} (with values in essentially small spaces). By \cite[Thm 5.1.9]{dagX}, it means that the map $\sigma_\ccal$ has $0$-truncated fibers, or in other words, that the map $\sigma_\ccal$ induces an isomorphism on $\pi_i$ for $i\geq 2$ and is injective on $\pi_1$ (all homotopy groups are based at the trivial deformation of $\ccal$). 

\item Let $\xi(\ccal)$ denote the $\et$-algebra over $k$ given by the center of the $\oo$-category $\ccal$ in the sense of \cite[Def 5.3.10]{dagX}. The underlying $k$-module spectrum of $\xi(\ccal)$ is given by $\uend_{\End_k(\ccal)} (\id_\ccal)$ so that the homotopy groups of $\xi(\ccal)$ are isomorphic to the Hochschild cohomology groups of $\ccal$. By \cite[Thm 3.5.1]{dagX}, the augmented $\et$-algebra which controls the $\et$-formal moduli problem $\catdefa_\ccal$ is given by $k\oplus \xi(\ccal)$. In other words, for any artinian $\et$-algebra $A$ over $k$ there exists a natural equivalence 
$$\catdefa_\ccal(A)\simeq \map_{\algetk} (\dt(A), \xi(\ccal))$$
where $\dt$ is the $\et$-Kozsul duality functor. In particular the tangent spectrum of $\catdefa_\ccal$ is given by $T_{\catdefa_\ccal}\simeq \xi(\ccal)[2]$. 
\end{itemize}

We have stronger properties if we restrict to compactly generated deformations as observed by Lurie (see \cite[Variant 5.3.4]{dagX}), a condition which is reasonable in most examples of interest. Let $\catdefc_\ccal$ denote the subfunctor of $\catdef_\ccal$ spanned by pairs $(\ccal_A,u)$ where $\ccal_A$ is compactly generated. By \cite[Variant 5.3.6]{dagX}, the functor $\catdefc_\ccal$ has values in essentially small spaces, and can therefore be viewed as a functor 
$$\catdefc_\ccal : \algetkart\lmo \spaces$$
which satisfies the following properties. 

\begin{itemize}
\item The natural embedding $\catdefc_\ccal\mo \catdef_\ccal$ has $0$-truncated fibers, which implies that $\catdefc_\ccal$ is a $2$-proximate $\et$-formal moduli problem and that the induced map $L(\catdefc_\ccal)\mo L(\catdef_\ccal)=\catdefa_\ccal$ is an equivalence. We denote by 
$$\theta_\ccal : \catdefc_\ccal\mo \catdefa_\ccal$$
the unit map. 

\item Suppose that $\ccal$ is tamely compactly generated (see Definition \ref{deftamely}). Then the functor $\catdefc_\ccal$ is a $1$-proximate $\et$-formal moduli problem by \cite[Prop 5.3.21]{dagX}. In other words, the map $\theta_\ccal$ induces an isomorphism on $\pi_i$ for $i\geq 1$ and an injection on $\pi_0$. Our main result gives additional conditions under which the surjectivity on $\pi_0$ holds for formal deformations. 
\end{itemize}

We recall now the construction of the functor $\catdef_\ccal$ from \cite[Const 5.3.2]{dagX}. By \cite[Thm 4.8.5.16]{ha} there exists a functor
\begin{align*}
\algk & \lmo \catk\\
A & \longmapsto \lmod_A
\end{align*} 
which extends to a symmetric monoidal functor. Passing to algebra objects we obtain a functor 
\begin{align*}
\algetk\simeq \alg(\algk) & \lmo \alg(\catk) \\
A & \longmapsto \lmod_A^\te
\end{align*}
which to an $\et$-algebra $A$ associates the monoidal $\infty$-category of left $A$-modules. 

Let $A$ be an $\et$-algebra over $k$. A \emph{right $A$-linear $\infty$-category} (resp. \emph{left $A$-linear $\infty$-category}) is by definition a right $\lmod_A^\te$-module in $\catk$ (resp. a left $\lmod_A^\te$-module in $\catk$). If $A\mo A'$ is a morphism of $\et$-algebras over $k$, and $\ccal_A$ a right $A$-linear $\infty$-category, we denote by $\ccal_A\te_A A'$ the extension of scalars $\ccal_A\tec_{\lmod_A} \lmod_{A'}$.

Consider the cocartesian fibration $\rmod(\prlk)\mo \alg(\prlk)$ with fiber over a $k$-linear monoidal $\oo$-category $\dal^\te$ given by the $\oo$-category $\rmod_{\dal^\te} (\prlk)$ of $\oo$-categories right tensored over $\dal^\te$. Set $\rcat_k=\algetk\times_{\alg(\prlk)} \rmod(\prlk)$ and consider the induced cocartesian fibration $p:\rcat_k\mo \algetk$ with fiber over an $\et$-algebra $A$ given by the $\oo$-category $\rmod_{\lmod_A^\te} (\prlk)$ of right $A$-linear $\oo$-categories. Consider the subcategory $\rcat_k^{\mrm{cocart}}$ consisting of $p$-cocartesian morphisms so that we obtain a left fibration $\rcat_k^{\mrm{cocart}}\mo \algetk$ with fiber over $A$ the space of right $A$-linear $\oo$-categories. 

Let $\ccal$ be a $k$-linear $\oo$-category and consider the pair $(k,\ccal)$ as an object of $\rcat_k^{\mrm{cocart}}$. The \emph{$\oo$-category of deformations of $\ccal$} is the $\oo$-category 
$\deform[\ccal]=(\rcat_k^{\mrm{cocart}})_{/(k,\ccal)}$. The induced left fibration $\deform[\ccal]\mo (\algetk)_{/k}\simeq \algetka$ is classified by a functor $\algetka\mo \bspaces$ whose restriction to artinian $\et$-algebras is denoted by $\catdef_\ccal : \algetkart\lmo \bspaces$.

\section{Generators and curvature in deformations of categories}\label{ssformaldef}

In this section we first set up some preliminary results concerning simultaneous deformations and deformations of $\eo$-algebras. Finally, in the last subsection, we use them to prove our main results concerning formal deformations.


\subsection{Simultaneous deformations}\label{sssimul}

Let $k$ be a field. We now study the deformation functor of a pair $(\ccal, E)$ with $\ccal$ a $k$-linear $\oo$-category and $E\in \ccal$ an object. 

We first give a construction of this functor which consists in mimicking the construction of $\catdef_\ccal$ in \cite[Const 5.3.2]{dagX} but replacing the $\oo$-category $\prlk$ by the slice $\oo$-category $(\prlk)_{\modk/}$ of pairs $(\ccal, E)$ with $\ccal$ a $k$-linear $\oo$-category and $E$ an object of $\ccal$. 

\begin{const}\label{constsimdef}

\cite[Thm 4.8.5.16]{ha} provides a symmetric monoidal functor 
\begin{align*}
\algk & \lmo (\prlk)_{\modk/}  \\
A & \longmapsto (\lmod_A, A)
\end{align*} 
where the $\oo$-category $(\prlk)_{\modk/}$ is endowed with the natural symmetric monoidal structure from \cite[Rem 2.2.2.25]{ha} coming from the fact that $\modk$ is the unit in $\prlkte$. 
Applying algebra objects we obtain a functor 
\begin{align*}
\algetk\simeq \alg(\algk) & \mo \alg((\catk)_{\modk/})\simeq \alg(\prlk) \\
A & \longmapsto \lmod_A^\te.
\end{align*} 
Consider the cocartesian fibration
$$\rmod(\prlkp)\mo \alg(\prlk)$$
with fiber over a monoidal $\oo$-category $\dal^\te$ the $\oo$-category $\rmod_{\dal^\te}(\prlkp)$ of pairs $(\M, M)$ where $\M$ is a $k$-linear $\oo$-category right tensored over $\dal^\te$ and $M\in \M$ an object. 
Set $\rcatkp=\algetk\times_{\alg(\prlk)} \rmod(\prlkp)$ and consider the induced cocartesian fibration $q:\rcatkp\mo \algetk$ with fiber over an $\et$-algebra $A$ the $\oo$-category of right $A$-linear $\oo$-categories together with a distinguished object. Let $(\rcatkp)^{\mrm{cocart}}$ be the subcategory of $q$-cocartesian morphisms so that we have an induced left fibration $(\rcatkp)^{\mrm{cocart}}\mo \algetk$. Let $\ccal$ be a $k$-linear $\oo$-category and $E$ an object of $\ccal$. The \emph{$\infty$-category of deformations of $(\ccal, E)$} is the $\oo$-category 
$$\deform[\ccal, E] := (\rcatkp)^{\mrm{cocart}}_{/(k,\ccal, E)}.$$
The induced left fibration $\deform[\ccal, E] \mo (\algetk)_{/k}\simeq \algetka$ is classified by a functor $\algetka\mo \bspaces$ whose restriction to artinian $\et$-algebras is denoted by 
$$\simdef_{(\ccal, E)} : \algetkart\lmo \bspaces.$$ 

For any artinian $\et$-algebra $A$ over $k$, the space $\simdef_{(\ccal, E)}(A)$ classifies $4$-uples $(\ccal_A, E_A,u,v)$ where $\ccal_A$ is a right $A$-linear $\oo$-category, $u:\ccal_A\te_A k\simeq \ccal$ is an equivalence, $E_A\in \ccal_A$ is an object and $v:E_A\te_A k\simeq A$ is an equivalence in $\ccal$. We have that $\simdef_{(\ccal, E)}(k)$ is a contractible space. 
\end{const}

The deformation functor $\simdef_{(\ccal, E)}$ is closely related to the deformation functors $\catdef_\ccal$ and $\objdef_E$. We construct the evident maps between them. 

\begin{const}\label{constmapssimdef} Let $k$ be a field, $\ccal$ a $k$-linear $\oo$-category and $E$ an object of $\ccal$. The forgetful functor $\prlkp\mo \prlk$ induces a functor $\rcatkp\mo \rcat_k$ which in turn induces a functor $\deform[\ccal, E]\mo\deform[\ccal]$ commuting with the projection to $\algetka$. We obtain natural transformations
$$\tau_{(\ccal, E)}:\simdef_{(\ccal, E)}\lmo \catdef_\ccal$$
$$\rho_{(\ccal, E)} : \simdefc_{(\ccal, E)}\lmo \catdefc_\ccal$$
given on objects by forgetting the distinguished object deformation. 

The homotopy fiber of the forgetful functor $\prlkp\mo \prlk$ at $\ccal$ is equivalent to the Kan complex $\map_{\prlk} (\modk, \ccal)\simeq \ccal^\simeq$. The induced functor $\ccal^\simeq \mo \prlkp$ is given on object by sending an object $E$ of $\ccal^\simeq$ to the pair $(\ccal, E)$. This functor induces a functor $\deform[E]\mo \deform[\ccal, E]$ commuting with the projection to $\algetka$. We therefore obtain a natural transformation 
$$\varphi_{(\ccal, E)}:\objdefet_E\lmo \simdef_{(\ccal, E)}$$
where $\objdefet$ is the restriction of the functor $\objdef_E$ to artinian $\et$-algebras over $k$ along the forgetful functor $\algetkart\mo \algkart$. The functor $\varphi_{(\ccal, E)}$ is given on object by sending a deformation $E_A$ of $E$ over an $\et$-algebra $A$ to the pair $(\ccal\te_k A, E_A)$.  Moreover, whenever $\ccal$ is compactly generated, we have that the map $\varphi_{(\ccal, E)}$ factors as a map $\objdefet_E\mo \simdefc_{(\ccal, E)}$. 
\end{const}

\begin{prop}\label{propfibseq}
Let $k$ be a field, $\ccal$ a $k$-linear $\infty$-category and $E\in \ccal$ an object. The sequence of natural transformations
$$\objdefet_E\lmo \simdef_{(\ccal, E)} \lmo \catdef_\ccal$$
from Construction \ref{constmapssimdef} is a fiber sequence in $\fun_*(\algetkart, \bspaces)$. Consequently, the functor $\simdef_{(\ccal, E)}$ is a $2$-proximate formal $\et$-moduli problem (in a larger universe). Moreover it induces a homotopy commutative diagram
$$\xymatrix{\objdefet_E\ar[r] \ar[d] & \simdef_{(\ccal, E)}\ar[d] \ar[r] &  \catdef_\ccal \ar[d] \\ \objdefeta_E\ar[r] &  \simdefa_{(\ccal, E)} \ar[r] & \catdefa_\ccal}$$
where the bottom line is a fiber sequence in $\fmpet$. 
\end{prop}

\begin{proof} 
The first assertion follows from the fact that the homotopy fiber of the functor $\prlkp\mo \prlk$ at $\ccal$ is given by the Kan complex $\ccal^\simeq$. The second assertion follows directly from the existence of the fiber sequence and the fact that $\objdefet_E$ and $\catdef_\ccal$ are $1$-proximate and $2$-proximate respectively. 
The last assertion follows from the fact that the three functors are $m$-proximate formal moduli problems for some $m$ and from the fact (after suitable enlargement of the universe) that the functor $L_m:\prox(m)_{\et}\mo \fmpet$ preserves small limits by Corollary \ref{corassfmp}. 
\end{proof}

\begin{rema}
We can already deduce from Proposition \ref{propfibseq} the tangent spectrum of the formal $\et$-moduli problem $\simdefa_\ccal$. However, to avoid dealing with the $\et$-algebra associated to the formal $\et$-moduli problem $\objdefeta_E$, we give a direct construction of the $\et$-algebra associated to $\simdefa_\ccal$, which will be useful in the sequel. 
\end{rema}

\begin{nota}\label{notacenterrel}
Let $k$ be a field, $\ccal$ a $k$-linear $\infty$-category and $E\in \ccal$ an object. Consider the object $(\ccal, E)$ of the slice $\oo$-category $\prlkp$ and its endomorphism monoidal $k$-linear $\oo$-category $\End_k(\ccal,E)^\te:=\End_{\prlkp} ((\ccal, E))^\te$. Denote by $\xi(\ccal, E)=\uend_{\End_k(\ccal,E)}(\id_{(\ccal, E)})$ the $\et$-algebra over $k$ given by endomorphisms of the unit. By adjunction, this $\et$-algebra satisfies the following universal property: for any $\et$-algebra $A$ over $k$, there exits a homotopy equivalence
$$\map_{\algetk} (A, \xi(\ccal, E))\simeq \lmod_{\lmod_A^\te} (\prlkp)^\simeq \times_{(\prlkp)^\simeq } \{(\ccal, E)\}$$
where the right hand side is the space of left $A$-linear structures on the pair $(\ccal, E)$, or in other words, the space of left $A$-linear structures on $\ccal$ acting trivially on $E$ up to equivalence.  
\end{nota}

\begin{rema}\label{remaendce}
Let $k$ be a field, $\ccal$ a $k$-linear $\infty$-category and $E\in \ccal$ an object. By definition of the slice $\oo$-category $\prlkp$, the $k$-linear $\oo$-category $\End_k(\ccal, E)$ is the homotopy fiber of the $k$-linear functor $\ev_E:\End(\ccal)\mo \ccal$ given by evaluation at $E$. We deduce that the induced sequence 
$$\xi(\ccal,E)\mo \xi(\ccal)\mo \uend_\ccal(E)$$
is a fiber sequence of $k$-module spectra. We observe that the first map in the sequence is a map of $\et$-algebras over $k$.  
\end{rema}

\begin{prop}\label{propsimdefa}
Let $k$ be a field, $\ccal$ a $k$-linear $\infty$-category and $E\in \ccal$ an object. Consider the augmented $\et$-algebra given by $k\oplus \xi(\ccal, E)$ and its associated $\et$-formal moduli problem $\Psi_{k\oplus \xi(\ccal, E)}$. Then there exists a natural transformation $\simdef_{(\ccal, E)}\mo \Psi_{k\oplus \xi(\ccal, E)}$ which is $0$-truncated. Therefore the formal $\et$-moduli problem $\simdefa_{(\ccal, E)}$ is given by the formula
$$\simdefa_{(\ccal, E)}(A)\simeq \map_{\algetk} (\dt(A), \xi(\ccal, E)).$$

\end{prop} 

\begin{proof} 
We mimic the proof of \cite[Thm 5.3.16]{dagX}, replacing the center $\xi(\ccal)$ by $\xi(\ccal, E)$. Following the same lines as in \cite[Const 5.3.18]{dagX}, we see that there exits a homotopy commutative diagram 
$$\xymatrix{ \deform[\ccal,E] \ar[r]^-{ } \ar[d]_-{ } & (\algetka)_{/\xi(\ccal, E)}  \ar[d]^-{ } \\  (\algetka)^\op \ar[r]^-{\dt } & \algetka  }$$
where the upper functor sends a deformation $(\ccal_A, E_A)$ over an augmented $\et$-algebra $A$ to $\ccal$ seen as left tensored over $\dt(A)$ via the equivalence $(\dt(A)\te_k \ccal_A)\te_{\dt(A)\te_k A} k \simeq \ccal$ together with the distinguished object $E$. This functor induces a natural transformation 
$$\simdef_{(\ccal, E)}\mo \Psi_{k\oplus \xi(\ccal, E)}.$$
The same proof as \cite[Prop 5.3.17]{dagX} but replacing the center $\xi(\ccal)$ by the $\et$-algebra $\xi(\ccal, E)$ shows that this map is $0$-truncated. Finally \cite[Thm 5.1.9]{dagX} implies that the induced map $\simdefa_{(\ccal, E)}\mo \Psi_{k\oplus \xi(\ccal, E)}$ is an equivalence. 
\end{proof}

\begin{rema}
In the situation of Proposition \ref{propsimdefa}, the universal property of the $\et$-algebra $\xi(\ccal, E)$ (see Notation \ref{notacenterrel}) implies that we have for each artinian $\et$-algebra $A$ over $k$ a natural homotopy equivalence
$$\simdefa_{(\ccal, E)}(A)\simeq \lmod_{\lmod_{\dt(A)}^\te} (\prlkp)^\simeq \times_{(\prlkp)^\simeq } \{(\ccal, E)\}$$
with the space of left $\dt(A)$-linear structures on the $\oo$-category $\ccal$ such that $\dt(A)$ acts trivially on the object $E$. 
\end{rema}

\begin{rema}\label{remaobsmap}
Let $k$ be a field, $\ccal$ a $k$-linear $\infty$-category and $E\in \ccal$ an object. Recall from Proposition \ref{propfibseq} the existence of a fiber sequence of formal $\et$-moduli problems 
$$\objdefeta_E\lmo \simdefa_{(\ccal, E)}\lmo \catdefa_\ccal.$$
It induces a fiber sequence 
$$\Omega \simdefa_{(\ccal, E)} \lmo \Omega \catdefa_\ccal \lmo \objdefeta_E$$
which after taking the tangent spectrum gives a fiber sequence of $k$-module spectra  
$$\xi(\ccal, E)\lmo \xi(\ccal)\lmo \uend_\ccal(E).$$
It can be seen that for each artinian $\et$-algebra $A$ over $k$, the map $\Omega \catdefa_\ccal(A) \mo \objdefeta_E(A)$ sends an $A$-linear autoequivalence $f_A: \ccal\te_k A\simeq \ccal\te_k A$ of the trivial deformation to the image $f_A(E\te_k A)$ of the trivial deformation of $E$ in $\ccal\te_k A$. From this we can deduce that the fiber sequence of tangent spectra above is equivalent to the fiber sequence given in Remark \ref{remaendce}, although we will not use this fact. We denote by $\chi_E : \xi(\ccal)\mo \uend_\ccal (E)$ the map of $k$-modules induced by evaluation at $E$ and call it the obstruction map of $E$. 
\end{rema}

\begin{rema}\label{remaobs}
Let $k$ be a field, $\ccal$ a $k$-linear $\infty$-category and $E\in \ccal$ an object. Consider the homotopy commutative diagram

$$\xymatrix{ \catdef_\ccal(k\oplus k[m])  \ar[r]^-{ } \ar[d]_-{ \sigma_\ccal} & \Omega\catdef_\ccal(k\oplus k[m+1] ) \ar[d]^-{ } \ar[r] & \objdefet_E(k\oplus k[m+1]) \ar[d] \\  \catdefa _\ccal(k\oplus k[m])  \ar[r]^-{\sim } & \Omega\catdefa_\ccal(k\oplus k[m+1] ) \ar[r] & \objdefeta_E(k\oplus k[m+1]) }$$

where the right square is induced by Proposition \ref{propfibseq}. The bottom left map is an equivalence because $\catdefa_\ccal$ is a formal moduli problem. By \cite[Thm 3.5.1]{dagX}, there exists an isomorphism $\pi_0\catdefa _\ccal(k\oplus k[m])\simeq \hh^{m+2}(\ccal)$. Moreover by \cite[Thm 5.2.8]{dagX} there is an isomorphism $\pi_0  \objdefa_E(k\oplus k[m+1])\simeq \extg^{m+2}_\ccal(E,E)$. From Proposition \ref{propfibseq} and taking $\pi_0$, we have an induced commutative diagram
$$\xymatrix{ \pi_0\simdef_{(\ccal, E)}(k\oplus k[m])\ar[r]^-{\alpha} \ar[d] & \pi_0\catdef_\ccal(k\oplus k[m])  \ar[r]^-{ } \ar[d]_-{ \sigma_\ccal} & \pi_0 \objdefet_E(k\oplus k[m+1]) \ar[d] \\   \pi_0\simdefa_{(\ccal, E)}(k\oplus k[m])\ar[r] &\hh^{m+2}(\ccal) \ar[r]^-{\chi_E} & \extg^{m+2}_\ccal(E,E) }$$
where the bottom line is exact in the middle. Let $\ccal_1$ be a deformation of $\ccal$ over $k\oplus k[m]$ and let $\phi=\sigma_\ccal(\ccal_1)$. Suppose $E_1\in \ccal_1$ deforms $E$ or in other words that $E_1\te_{k\oplus k[m]} k\simeq E$ in $\ccal_1\te_{k\oplus k[m]} k\simeq \ccal$. In this case we have $\alpha(\ccal_1, E_1)=\ccal_1$ and thus $\chi_E(\phi)=0$. 
\end{rema}

\begin{ex}\label{exarigid}
We can now explain a simple example in which $\catdef_\ccal$ fails to be a formal moduli problem. It is taken from \cite[Ex 3.14]{kellow}. Let $k$ be a field of characteristic zero and $A=k[u,u^{-1}]$ the free graded commutative $k$-algebra generated by a degree $2$ variable $u$ and its inverse over $k$. We regard $A$ as an $\ei$-algebra and let $\ccal=\Mod_A$ be the $k$-linear $\infty$-category of $A$-module spectra. A computation gives $\hh^2(\ccal)\simeq k$ generated by the multiplication by $u$ whose corresponding cocycle is denoted by $\phi_u$. Then for every non zero $A$-module $M$, we have that $\chi_M(\phi_u)\neq 0$ in $\extg^2_A(M,M)$ because multiplication by $u$ on $M$ is an equivalence. Suppose now that $\phi_u$ is the image by the map $\sigma_\ccal:\pi_0\catdef_\ccal(k[t]/t^2)\mo \hh^2(\ccal)$ of a deformation $\ccal_1$ of $\ccal$ over $k[t]/t^2$. Then any $M_1\neq 0$ in $\ccal_1$ satisfies $M=M_1\te_{k[t]/t^2} k\neq 0$ in $\ccal$, and by Remark \ref{remaobs} we have $\chi_M(\phi_u)=0$, which is a contradiction. We deduce that $\phi_u$ is not in the image of the map $\sigma_\ccal:\pi_0\catdef_\ccal(k[t]/t^2)\mo \hh^2(\ccal)$ and that $\catdef_\ccal$ is not a formal $\et$-moduli problem. 
\end{ex}

\begin{nota}\label{notadefoc}
Let $k$ be a field, $\ccal$ a compactly generated $k$-linear $\infty$-category and $E\in \ccal$ an object. For each artinian $\et$-algebra $A$ over $k$ we denote by $\simdefc_{(\ccal, E)}(A)$ the summand of $\simdef_{(\ccal, E)}(A)$ consisting of $4$-uplets $(\ccal_A,u,E_A,v)$ with $\ccal_A$ being compactly generated. We obtain a functor 
$$\simdefc_{(\ccal, E)}:\algetkart\mo \bspaces.$$
\end{nota}

\begin{prop}\label{propsimdefc}
Let $k$ be a field, $\ccal$ a compactly generated $k$-linear $\infty$-category and $E\in \ccal$ an object. The functor $\simdefc_{(\ccal, E)}$ is a $2$-proximate $\et$-moduli problem and the inclusion $\simdefc_{(\ccal, E)}\mo \simdef_{(\ccal, E)}$ induces an equivalence of formal $\et$-moduli problems $L(\simdefc_{(\ccal, E)})\simeq \simdefa_{(\ccal, E)}$. If moreover $\ccal$ is supposed to be tamely compactly generated, then $\simdefc_{(\ccal, E)}$ is a $1$-proximate formal $\et$-moduli problem. 
\end{prop}

\begin{proof} 
A variant of Proposition \ref{propfibseq} implies the existence of a homotopy commutative diagram 
$$\xymatrix{\objdefet_E\ar[r] \ar[d] & \simdefc_{(\ccal, E)}\ar[d] \ar[r] &  \catdefc_\ccal \ar[d] \\ \objdefeta_E\ar[r] &  \simdef^{c,\wedge}_{(\ccal, E)} \ar[r] & \catdefa_\ccal}$$
where the two horizontal top maps are constructed in Construction \ref{constmapssimdef}. This diagram implies the first two assertions. The last one follows from this diagram and the fact that if $\ccal$ is tamely compactly generated then by \cite[Prop 5.3.21]{dagX} the functor $\catdefc_\ccal$ is $1$-proximate. 
\end{proof}

\begin{rema}
Let $k$ be a field, $\ccal$ a tamely compactly generated $k$-linear $\infty$-category and $E\in \ccal$ an object. Let $m\geq 0$ be an integer and consider the following commutative diagram induced by the commutative diagram of the proof of Proposition \ref{propsimdefc}, similarly to Remark \ref{remaobs}.
$$\xymatrix{ \pi_0\simdefc_{(\ccal, E)}(k\oplus k[m])\ar[r] \ar[d] & \pi_0\catdefc_\ccal(k\oplus k[m])  \ar[r]^-{} \ar[d]_-{ \theta_\ccal} & \pi_0 \objdefet_E(k\oplus k[m+1]) \ar[d] \\   \pi_0\simdefa_{(\ccal, E)}(k\oplus k[m])\ar[r] &\hh^{m+2}(\ccal) \ar[r]^-{\chi_E} & \extg^2_\ccal(E,E) }$$
The bottom line is exact in the middle. All vertical maps are injective because when $\ccal$ is tamely compactly generated, the functors $\simdefc_{(\ccal, E)}$ and $\catdefc_\ccal$ are $1$-proximate and moreover $\objdefet_E$ is $1$-proximate. This implies that the upper line is also exact in the middle, as a sequence of pointed sets. We deduce from this that given a compactly generated deformation $\ccal_1$ of $\ccal$ over $k\oplus k[m]$, there exists an object $E_1\in \ccal_1$ deforming $E$ if and only if $\chi_E(\theta_\ccal(\ccal_1))=0$. 
\end{rema}


\subsection{Deformations of associative algebras}\label{ssass}

We recall some known facts about the deformation functor of an associative algebra and construct maps to the formal $\et$-moduli problem of its $\oo$-category of left modules and from the formal $\et$-moduli problem of simultaneous deformations of pairs $(\ccal, E)$ given by taking endomorphisms. 

\begin{const}\label{constalgdef}
Consider the functor $\algetk\mo \alg(\prlk)$ given on objects by $A\mapsto\lmod_A^\te$ (see \cite[Not 5.3.1]{dagX} or §\ref{sscatdef}). For an $\et$-algebra $A$ over $k$, we denote by $\alg_A=\alg(\lmod_A^\te)$ the presentable $\oo$-category of $A$-algebras. Composing the functor $\algetk\mo \alg(\prlk)$ with the functor $\alg : \alg(\prlk)\mo \prl$ given by taking algebra objects, we obtain a functor $\algetk\mo \prl$ sending an $\et$-algebra $A$ to the $\oo$-category $\alg_A$ of $A$-algebras. This functor is classified by a cocartesian fibration $s:\alg\mo \algetk$ with fiber over an $\et$-algebra $A$ given by the $\oo$-category $\alg_A$. Consider the subcategory of $s$-cocartesian morphisms $\alg^{\mrm{cocart}}$ and the induced left fibration $\alg^{\mrm{cocart}}\mo \algetk$. 

Let $B$ be an $\eo$-algebra over $k$. The $\oo$-category of deformation of $B$ is the $\oo$-category $\deform[B]=(\alg^{\mrm{cocart}})_{/(k,B)}$. The induced left fibration $\deform[B]\mo \algetka$ is classified by a functor $\algetka\mo \bspaces$ whose restriction to artinian $\et$-algebras is denoted by 
$$\algdef_B : \algetkart\lmo \bspaces$$
For any artinian $\et$-algebra $A$, the space $\algdef_B(A)$ classifies pairs $(B_A,u)$ with $B_A$ an $A$-algebra and $u:B_A\te_A k \simeq B$ an equivalence of $\eo$-algebras over $k$. 
\end{const} 

\begin{prop}\label{propalgdefprox}
Let $k$ be a field and $B$ an $\eo$-algebra over $k$. For every pullback diagram 
$$\xymatrix{ A \ar[r]^-{ } \ar[d]_-{ } & A_1 \ar[d]^-{ } \\ A_0 \ar[r]^-{ } & A_{01} }$$
in $\algetkart$, the induced map 
$$\algdef_B(A)\mo \algdef_B(A_0)\times_{\algdef_B(A_{01})} \algdef_B(A_1)$$ 
has $(-1)$-truncated homotopy fibers (or equivalently induces a homotopy equivalence onto its essential image). 
\end{prop} 

\begin{proof} 
By \cite[Thm 7.4]{dagIX}, for every such pullback square in $\algetkart$, the induced functor 
$$\lmod_A\mo \lmod_{A_0}\times_{\lmod_{A_{01}}} \lmod_{A_1}$$
is fully faithful. Because the functor $\alg: \alg(\prl)\mo \prl$ given by applying algebra objects preserves small limits and preserves fully faithful functors, the induced functor
$$\alg_A\mo \alg(\lmod_{A_0}\times_{\lmod_{A_{01}}} \lmod_{A_1})\simeq \alg_{A_0}\times_{\alg_{A_{01}}} \alg_{A_1}$$
is fully faithful, which implies our assertion. 
\end{proof}

\begin{cor}\label{coralgdefsm}
Let $k$ be a field and $B$ an $\eo$-algebra over $k$. For every artinian $\et$-algebra $A$ over $k$, the space $\algdef_B(A)$ is essentially small. Therefore $\algdef_B$ can regarded as a functor with values in the $\oo$-category $\spaces$ of small spaces. 
\end{cor} 

\begin{proof} 
Let $m\geq 0$ be an integer. By Proposition \ref{propalgdefprox} the map $\algdef_B(k\oplus k[m])\mo \Omega\algdef_B(k\oplus k[m+1])$ induces a homotopy equivalence onto its essential image. Moreover the target of this map is homotopy equivalent to the homotopy fiber of the restriction map $\map_{\alg_A}(B\te_k A, B\te_k A)\mo \map_{\algk} (B,B)$ at $\id_B$ and is therefore essentially small. This implies that $\algdef_B(k\oplus k[m])$ is essentially small for any $m\geq 0$. Let now $A$ be an artinian $\et$-algebra and let $A=A_0\mo A_1\mo \hdots \mo A_n=k$ be a finite sequence of maps such that for every $0\leq i\leq n-1$ there exists an integer $m_i\geq 0$ and a pullback diagram
$$\xymatrix{ A_{i}  \ar[r]^-{ } \ar[d]_-{ } & A_{i+1}  \ar[d]^-{ } \\ k \ar[r]^-{ } & k\oplus k[m_i].}$$
By Proposition \ref{propalgdefprox} the induced map 
$$\algdef_B(A)\mo \algdef_B(k)\times_{\algdef_B(k\oplus k[m_i])} \algdef_B(A_{i+1})$$ 
is a homotopy equivalence onto its essential image. We then deduce by descending induction on $i$, using the first part of the proof that $\algdef_B(A)$ is essentially small. 
\end{proof}

\begin{cor} 
Let $k$ be a field and $B$ an $\eo$-algebra over $k$. The the functor $\algdef_B:\algetkart\mo \spaces$ is a $1$-proximate formal $\et$-moduli problem. 
\end{cor}

\begin{proof} 
It follows by combining Proposition \ref{propalgdefprox} and Corollary \ref{coralgdefsm}. 
\end{proof}

\begin{prop}\label{propalgdefnfmp}
Let $k$ be a field and let $B$ be an $n$-connective $\eo$-algebra over $k$ for some integer $n$ (i.e. the underlying $k$-module of $B$ is $n$-connective). Then the functor $\algdef_B$ is a formal $\et$-moduli problem. 
\end{prop} 

\begin{proof} 
For an $\et$-algebra $A$, let $\lmod_A^{\geq n}$ denote the full subcategory of $\lmod_A$ spanned by $n$-connective left $A$-modules. Moreover let $\alg_A^{\geq n}$ denote the full subcategory of $\alg_A$ spanned by the $n$-connective $A$-algebras. By a variant of \cite[Prop 7.6]{dagIX}, for every pullback diagram 
$$\xymatrix{ A \ar[r]^-{ } \ar[d]^-{ } & A_0 \ar[d]^-{ } \\ A_1 \ar[r]^-{ } & A_{01} }$$
in $\algetkart$ such that the maps $\pi_0 A_0\mo \pi_0 A_{01} \leftarrow \pi_0 A_1$ are surjective, the induced functor 
$$\lmod_A^{\geq n} \lmo \lmod_{A_0}^{\geq n} \times_{ \lmod_{A_{01}}^{\geq n} } \lmod_{A_1}^{\geq n} $$ 
is an equivalence, where $\lmod_A^{\geq n}$ is the $\oo$-category of $n$-connective left $A$-modules. Because for every $\et$-algebra $A'$ the forgetful functor $\alg_{A'}\mo \lmod_{A'}$ is conservative, the proof of \cite[Prop 7.6]{dagIX} also shows that the induced functor
$$\alg_A^{\geq n} \lmo \alg_{A_0}^{\geq n}\times_{\alg_{A_{01}}^{\geq n}} \alg_{A_1}^{\geq n}$$
is an equivalence. Remark that if $M$ is an $n$-connective $k$-module and if $M_{A'}$ is a deformation of $M$ over an artinian $\et$-algebra $A'$, then the $A'$-module $M_{A'}$ is $n$-connective (see the proof of \cite[Prop 5.2.14]{dagX}). By the above, this implies that the induced map 
$$\algdef_B(A)\lmo  \algdef_B(A_0)\times_{\algdef_B(A_{01})} \algdef_B(A_1)$$
is a homotopy equivalence and therefore that $\algdef_B$ is a formal $\et$-moduli problem. 
\end{proof}

\begin{nota}\label{notamoduledef} 
Let $k$ be a field and $B$ an $\eo$-algebra over $k$. Consider the cocartesian fibration $s:\alg\mo \algetk$ from Construction \ref{constalgdef} and the cocartesian fibration $p:\rcatk\mo \algetk$ from §\ref{sscatdef}. By \cite[Not 4.8.5.10]{ha} there exists a functor $\alg\mo \rcatk$ which to a pair $(A,B_A)$ with $A$ an $\et$-algebra over $k$ and $B_A$ an $A$-algebra, assigns the pair $(A,\lmod_{B_A}(\lmod_A))$ where $\lmod_{B_A}(\lmod_A)$ is seen as right tensored over $\lmod_A^\te$. By \cite[Prop 4.8.5.1]{ha} the functor $\alg\mo \rcatk$ sends $s$-cocartesian morphisms to $p$-cocartesian morphisms and therefore induces a natural transformation $\algdef_B\mo \catdef_{\lmod_B}$ which factors as a natural transformation
$$\M_B : \algdef_B\lmo \catdefc_{\lmod_B}$$
given on objects by applying the $\oo$-category of left modules. 
\end{nota}

\begin{nota}
Let $k$ be a field, let $\ccal$ be a compactly generated $k$-linear $\oo$-category and let $E\in \ccal$ be an object. Consider the cocartesian fibration $s:\alg\mo \algetk$ from Construction \ref{constalgdef} and the cocartesian fibration $q:\rcatkp\mo \algetk$ from Construction \ref{constsimdef}. By \cite[Thm 4.8.5.11]{ha}, there exists a functor $\rcatkp\mo \alg$ which to a triple $(A,\ccal_A, E_A)$ with $A$ an $\et$-algebra over $k$, $\ccal_A$ a right $A$-linear $\oo$-category and $E_A\in \ccal_A$ an object, assigns the pair $(A,\uend_{\ccal_A} (E_A))$ where $\uend_{\ccal_A} (E_A)$ is seen as an $A$-algebra. It is easy to see that the functor $\rcatkp\mo \alg$ sends $q$-cocartesian morphisms to $s$-cocartesian morphisms and therefore induces a natural transformation 
$$\ecal_{(\ccal, E)} : \simdefc_{(\ccal, E)}\lmo \algdef_{\uend_\ccal(E)}$$ 
given on objects by taking endomorphisms. 
\end{nota}

\begin{rema}\label{remagendef}
Let $k$ be a field, let $\ccal$ be a tamely compactly generated $k$-linear $\oo$-category and let $E\in \ccal$ be a compact generator of $\ccal$. Let $A$ be an artinian $\et$-algebra over $k$, and suppose given a deformation $\ccal_A$ of $\ccal$ over $A$ and a deformation $E_A\in \ccal_A$ of $E$. Then by  \cite[Lem 5.3.31]{dagX} we have that $\ccal_A$ is tamely compactly generated, and \cite[Lem 5.3.37]{dagX} implies that the object $E_A$ is a compact generator of $\ccal_A$. Therefore by \cite[Thm 7.1.2.1]{ha} (Schwede--Shipley result), there exists a natural $A$-linear equivalence 
$$\lmod_{\uend_{\ccal_A}(E_A)}(\lmod_A)\simeq \ccal_A.$$
This fact implies that the diagram 
$$\xymatrix{ \simdefc_{(\ccal, E)} \ar[r]^-{\ecal_{(\ccal, E)} } \ar[rd]_-{\rho_{(\ccal, E)}} & \algdef_{\uend_\ccal (E)}  \ar[d]^-{\M_{\uend_\ccal (E)}} \\ & \catdefc_\ccal }$$
commutes up to natural homotopy, where $\rho_{(\ccal, E)}$ is the natural transformation from Construction \ref{constmapssimdef} given by forgetting the distinguished object. Indeed, unwinding the construction of each of these functors, we see that it suffices to produce for each artinian $\et$-algebra $A$ over $k$ an equivalence as above.
\end{rema}


\subsection{Formal deformations}

Let $k$ be a field and let $\ccal$ be a $k$-linear $\oo$-category. Recall that under the assumption that $\ccal$ is tamely compactly generated (see Definition \ref{deftamely}), we have by \cite[Prop 5.3.21]{dagX} that the natural transformation
$$\theta_\ccal : \catdefc_\ccal \lmo \catdefa_\ccal$$ 
is $(-1)$-truncated. Under some additional assumptions on $\ccal$, we now prove that the map $\theta_\ccal$ induces a surjection on $\pi_0$ for formal deformations, and therefore is an equivalence on formal deformations (that is, over the adic ring $\kfor$ of formal power series). To prove this, we prove a stronger statement which implies the existence of a compact generator for formal deformations.

We first set up some conventions about formal deformations. 

Let $n\geq 1$ be an integer, let $k$ be a field and let $X:\algenkart\mo \spaces$ be a functor. Consider the discrete commutative algebra $\kfor$ of formal power series in one variable as an augmented $\en$-algebra over $k$ via the forgetful functor $\calgka\mo \algenka$. The space of formal deformations in $X$ is by definition the space 
$$X(\kfor)=\map_{\fun_*(\algenkart, \spaces)} (\spf(\kfor), X)$$ 
where $\spf:\algenka\mo \fmpen$ is the formal spectrum (see Notation \ref{notaspf}).

\begin{lem}\label{lemdualdisalg}
Let $k$ be a field and let
$$\hdots \mo B_2 \mo  B_1 \mo  B_0$$
be an inverse system of discrete local artinian commutative $k$-algebras having same residue field $k$. Let $B=\lim_i B_i$ be the limit. Let $n\geq 2$ and consider for each $i$ the algebra $B_i$ as an artinian $\en$-algebra and $B$ as an augmented $\en$-algebra via the forgetful functor $\calgka\mo \algenka$. Then the natural map $\colim_i \spf(B_i)\mo \spf(B)$ is an equivalence of formal $\en$-moduli problems. Equivalently, the natural map $\colim_i \dn(B_i)\mo \dn(B)$ is an equivalence of augmented $\en$-algebras over $k$. 
\end{lem} 

\begin{proof} 
The statement for $n=\infty$ is given by \cite[Lem 6.3.3]{dagXII}. More precisely, let $R=k\oplus k[m]$ as an artinian $\ei$-algebra for some integer $m\geq 0$ and consider the homotopy commutative diagram 
$$\xymatrix{ \colim_i \map_{\calgka} (B_i, R)  \ar[r]^-{ } \ar[d]_-{ } & \map_{\calgka}(B,R) \ar[d]^-{ } \\  \colim_i \map_{\algenka} (B_i, R)  \ar[r]^-{ } & \map_{\algenka}(B,R) }$$
where the vertical maps are given by the forgetful functor $\calgka\mo \algenka$. By \cite[Lem 6.3.3]{dagXII} the horizontal top map is a homotopy equivalence. Denote by $\calgkah$ and $\algenkah$ the subcategories of discrete augmented $\ei$-algebras over $k$ and of discrete augmented $\en$-algebras over $k$ respectively. As $n\geq 2$ the forgetful functor induces an equivalence of categories $\calgkah\simeq  \algenkah$. The $\ei$-algebras $B_i$ and $B$ being discrete, this implies that the vertical maps in the diagram are homotopy equivalences. Therefore the horizontal bottom map is a homotopy equivalence as well. We have proven that the map $\colim_i \spf(B_i)\mo \spf(B)$ induces an equivalence on tangent spectra, which implies that it is an equivalence of formal $\en$-moduli problems. The last statement follows from Lemma \ref{lemalgrepspf} by applying the equivalence $\tfr^{(n)}:\fmpen\mo \algenka$. 
\end{proof}

\begin{rema}\label{remafordef}
Let $n\geq 2$ be an integer, let $k$ be a field and let $X:\algenkart\mo \spaces$ be a functor. Applying Lemma \ref{lemdualdisalg} to the ring $\kfor=\lim_i k[t]/t^i$ of formal power series in one variable over $k$, we obtain a natural homotopy equivalence: 
$$X(\kfor)\simeq \lim_i X(k[t]/t^i).$$
Therefore a point in $X(\kfor)$ is essentially given by a formal family $\{\eta_i\}_i$ of points $\eta_i\in X(k[t]/t^i)$ together with for every $i\geq 1$, a path $p_{i} (\eta_{i+1} )\simeq \eta_i$ in $ X(k[t]/t^i)$, where $p_{i} :X(k[t]/t^{i+1})\mo X(k[t]/t^i)$ is the map induced by the projection $k[t]/t^{i+1}\mo k[t]/t^{i}$, plus higher coherences. 
\end{rema}

\begin{rema}
Let $k$ be a field and let $\ccal$ be a compactly generated $k$-linear $\oo$-category. Remark \ref{remafordef} applied to the functor $\catdefc_\ccal:\algetkart\mo \spaces$ implies that a point in $\catdefc_\ccal(\kfor)$ is essentially given by a formal family $\{\ccal_i\}_i$ with $\ccal_i$ a compactly generated $k[t]/t^i$-linear $\oo$-category together with for each $i\geq 1$ a $k[t]/t^i$-linear equivalence $\ccal_{i+1} \te_{k[t]/t^{i+1}} k[t]/t^i\simeq \ccal_i$.
\end{rema}

\begin{rema}\label{remaku}
Let $k$ be a field and let $\ccal$ be a $k$-linear $\oo$-category. Consider the space of formal deformations $\catdefa(\kfor)=\map_{\fmpet}(\spf(\kfor), \catdefa)$. Applying the equivalence $\tfr^{(2)}:\fmpet\mo \algetka$, by Lemma \ref{lemalgrepspf} and \cite[Thm 3.5.1]{dagX} there exists natural homotopy equivalences
\begin{align*}
\catdefa(\kfor) & = \map_{\fmpet}(\spf(\kfor), \catdefa_\ccal) \\
 & \simeq \map_{\algetka} (\dt(\kfor), k\oplus \xi(\ccal)) \\
 & \simeq \map_{\algetk} (\dt(\kfor), \xi(\ccal)),
\end{align*}
where $\xi(\ccal)$ is the $\et$-algebra given by the center of $\ccal$. 

The usual computation using the Koszul--Tate resolution gives $\dt(\kfor)\simeq k[u]$ where $k[u]$ is the augmented $\et$-algebra over $k$ associated to the free augmented graded commutative $k$-algebra on one generator in cohomological degree $2$ via the forgetful functor. We deduce from this and from the universal property of the center (\cite[Rem 5.3.11]{dagX}) that there exists a natural homotopy equivalence 

$$\catdefa_\ccal (\kfor) \simeq \map_{\algetk} (k[u], \xi(\ccal)) \simeq \lmod_{\lmod_{k[u]}^\te} (\prlk)^\simeq \times_{(\prlk)^\simeq} \{\ccal\}$$

with the space of left $k[u]$-linear structures on the $\oo$-category $\ccal$. 
\end{rema}

Here is our first main result.

\begin{theo}\label{thmcatdeffor}
Let $k$ be a field and let $\ccal$ be a tamely compactly generated $k$-linear $\infty$-category which has a compact generator. Then the map
$$\theta_\ccal^t: \catdefc_\ccal(\kfor)\mo \catdefa_\ccal(\kfor)$$ 
induced by $\theta_\ccal$ is a homotopy equivalence. 
\end{theo}

We will deduce Theorem \ref{thmcatdeffor} from the following intermediate statement.

\begin{prop}\label{propgencatdef}
Let $k$ be a field and let $\ccal$ be tamely compactly generated $k$-linear $\infty$-category which has a compact generator $E\in \ccal$. Let $\alpha\in  \catdefa_\ccal(\kfor)$ be any point. Then there exists another compact generator $E^u$ of $\ccal$ and a formal deformation $B_t$ of its endomorphism algebra $B=\uend_\ccal (E^u)$ such that the connected component of $B_t$ is sent to the connected component of $\alpha$ by the composite map 
$$\pi_0\algdef_B(\kfor)\lmos{\M_B^t} \pi_0\catdefc_{\ccal}(\kfor) \lmos{\theta_\ccal^t} \pi_0\catdefa_\ccal(\kfor).$$
In particular, the map $\theta_\ccal^t : \catdefc_\ccal(\kfor)\mo\catdefa_\ccal(\kfor)$ is surjective on $\pi_0$. 
\end{prop}

\begin{proof}[Proof of Theorem \ref{thmcatdeffor}] 
By \cite[Prop 5.3.21]{dagX} the map $\theta_\ccal: \catdefc_\ccal \mo \catdefa_\ccal$ is $(-1)$-truncated. This implies that the induced map  
$$\theta_\ccal^t:\map_{\fmpet}  (\spf(\kfor), \catdefc_\ccal) \mo \map_{\fmpet}(\spf(\kfor), \catdefa_\ccal)$$ 
is $(-1)$-truncated as well. By Proposition \ref{propgencatdef} we have that the map $\theta_\ccal^t$ is surjective on $\pi_0$, which proves that it is a homotopy equivalence. 
\end{proof}

Here is another consequence of Proposition \ref{propgencatdef}. 

\begin{theo}\label{thmgenfordef}
Let $k$ be a field and let $\ccal$ be tamely compactly generated $k$-linear $\infty$-category which has a compact generator $E\in \ccal$. Let $\ccal_t=\{\ccal_i\}_i$ be a compactly generated formal deformation of $\ccal$. Then there exists another compact generator $E^u$ of $\ccal$ and a formal deformation $B_t=\{B_i\}_i$ of its endomorphism algebra $B=\uend_\ccal (E^u)$ such that for every $i\geq 1$ there exists a $k$-linear equivalence $\ccal_i\simeq \lmod_{B_i}$. In particular, each $\ccal_i$ is tamely compactly generated and admits a compact generator. 
\end{theo}

\begin{proof} 
Let $\ccal_t=\{\ccal_i\}_i$ be a compactly generated formal deformation of $\ccal$. Consider the vertex $\alpha=\theta_\ccal^t (\ccal_t)\in \catdefa_\ccal(\kfor)$. Let $E^u\in \ccal$ be the compact generator provided by Proposition \ref{propgencatdef} so that there exists a formal deformation $B_t=\{B_i\}_i$ of $B=\uend_\ccal (E^u)$ such that the connected component of $B_t$ is sent to the connected component of $\alpha$ by the map $\theta_\ccal^t\circ m_B^t$. But $\theta_\ccal^t$ is injective on $\pi_0$ by Theorem \ref{thmcatdeffor}, so that for every $i\geq 1$ there exists a $k$-linear equivalence $\lmod_{B_i}\simeq \ccal_i$. As $\ccal$ is tamely compactly generated, the $\eo$-algebra $B$ is $n_0$-connective for some integer $n_0$. By Proposition \ref{propalgdefnfmp} this implies that for every $i\geq 1$ the $\eo$-algebra $B_i$ is $n_0$-connective, and therefore that $\ccal_i$ is tamely compactly generated. 
\end{proof}

\begin{rema}\label{rematcgdef}
In the situation of Theorem \ref{thmgenfordef}, the fact that each $\ccal_i$ is tamely compactly generated for every compactly generated formal deformation $\ccal_t$ of $\ccal$ is already a consequence of \cite[Lem 5.3.31]{dagX} and does not require the assumption that $\ccal$ has a compact generator.  
\end{rema}

\begin{proof}[Proof of Proposition \ref{propgencatdef}]
Let $\alpha\in \catdefa_\ccal (\kfor)$. Recall from Remark \ref{remaku} the existence of a homotopy equivalence $\catdefa_\ccal(\kfor)\simeq \map_{\algetk} (k[u], \xi(\ccal))$ with the space of left $k[u]$-linear structures on $\ccal$. We denote by $\ccal_u$ the left $k[u]$-linear structure corresponding to $\alpha$. Let $E$ be a compact generator of $\ccal$. We exhibit another compact generator of $\ccal$ on which $k[u]$ acts trivially. The left $k[u]$-action gives a map of $k$-modules $k[u]\mo \uend_\ccal(E)$ and consider the map $\phi_E:E\mo E[2]$ given by the image of $u$. Define $E^u$ to be the homotopy fiber of $\phi_E$ in $\ccal$, it is well defined up to a contractible space of choice. We claim that the object $E^u$ is still a compact generator of $\ccal$. That it is compact follows from the fact that the subcategory $\ccal^c\subseteq \ccal$ spanned by compact objects is stable under finite colimits in $\ccal$ and that $\ccal$ is stable. To see that it is a generator, let $F$ be an object of $\ccal$ such that $\umap_\ccal(E^u,F)\simeq 0$ and we must prove that $F\simeq 0$. For this we can suppose that $F$ is compact. The fiber sequence $E^u\mo E\mo E[2]$ induces a fiber sequence of $k$-module spectra
$$\umap_\ccal(E,F)[-2] \mo \umap_\ccal (E,F)\mo \umap_\ccal (E^u,F)\simeq 0$$
which implies that $\umap_\ccal (E,F)\simeq \umap_\ccal (E,F)[-2]$. Because $\ccal$ is tamely compactly generated, we deduce that $\umap_\ccal (E,F)\simeq 0$. But $E$ is a generator of $\ccal$ and thus $F\simeq 0$. 

We adopt the notation 
$$\lmod_{\lmod_{k[u]}^\te} (\prlk)^\simeq_{\ccal}=\lmod_{\lmod_{k[u]}^\te} (\prlk)^\simeq\times_{(\prlk)^\simeq } \{\ccal\}$$ 
for the space of left $k[u]$-linear structures on $\ccal$ and similarly for left actions of $k[u]$ on the pair $(\ccal, E)$. 
Consider the diagram

$$\xymatrix{ \simdefc_{(\ccal, E^u)}(\kfor) \ar[r]^-{ } \ar[d]_-{\ecal} &  \simdefa_{(\ccal, E^u)}(\kfor) \ar[d]^-{ }  \ar[r]^-\sim  & \lmod_{\lmod_{k[u]}^\te} (\prlkp)^\simeq_{(\ccal,E^u)} \ar[dd]^-{\pi}   \\ \algdef_{\uend_\ccal (E^u)}(\kfor)  \ar[r]^-{\upsilon } \ar[d]_-{\M}  & \algdefa_{\uend_\ccal (E^u)}(\kfor) \ar[d] \\ \catdefc_\ccal(\kfor) \ar[r]^-{ \theta}  & \catdefa_\ccal(\kfor) \ar[r]^-\sim  &   \lmod_{\lmod_{k[u]}^\te} (\prlk)^\simeq_{\ccal}}$$

where the notations $\ecal$, $\M$, and $\theta$ have been removed indices for simplicity. By construction, the two smaller squares in this diagram commutes up to natural homotopy. Moreover the outer square commutes up to natural homotopy by virtue of Remark \ref{remagendef}. Therefore the right square is homotopy commutative as well, up to natural homotopy. The two horizontal maps in the right square are equivalences by \cite[Thm 5.3.16]{dagX} and Proposition \ref{propsimdefa}.  Because $\ccal$ is supposed to be tamely compactly generated and $E$ compact, the $\eo$-algebra $\uend_\ccal(E)$ is $n_0$-connective for some integer $n_0$ which implies that the map $\upsilon$ is a homotopy equivalence by Proposition \ref{propalgdefnfmp}. 

We claim that the object $(\ccal_u, E^u)$ of $\prlkp$ provides a lift of $\ccal_u$ through the map $\pi$. Unwinding the definitions, we must show that $E^u$ admits a lift through the functor $\fun_{\lmod_{k[u]}^\te} (\modk, \ccal_u)\mo \fun_{\lmod_{k[u]}^\te} (\lmod_{k[u]}, \ccal_u)\simeq \ccal_u$ induced by the monoidal functor $\lmod_{k[u]}^\te\mo \modk^\te$ itself induced by the augmentation map $k[u]\mo k$. By the right handed variant of \cite[Thm 4.8.4.1]{ha}, there exists an equivalence of $\oo$-categories $\fun_{\lmod_{k[u]}^\te} (\modk, \ccal_u)\simeq \rmod_k(\ccal_u)$. Therefore we must show that there exists a lift of $E^u$ through the forgetful functor $\rmod_k(\ccal_u)\mo \ccal_u$. By a variant of \cite[Cor 4.2.4.4]{ha} this forgetful functor admits a right adjoint $\ccal_u\mo \rmod_k(\ccal_u)$ given by the cofree right $k$-module functor which to an object $F\in \ccal_u$ assigns the right $k$-module given by the exponential $F^k$ of $F$ by the left $k[u]$-module $k$. There is a fiber sequence $k[u]\mo k[u][-2]\mo k$ in $\lmod_{k[u]}$, where the first map is multiplication by $u$. This implies the existence of a fiber sequence $E^k\mo E[2]\mo E$ in $\ccal_u$ and thus of an equivalence $E^k\simeq E^u$ in $\ccal_u$, which shows that $E^k$ is sent to $E^u$ via the forgetful functor $\rmod_k(\ccal_u)\mo \ccal_u$. 

We deduce that the pair $(\ccal_u, E^u)$ provides a well defined object in the $\oo$-category $\lmod_{\lmod_{k[u]}^\te} (\prlkp)_{(\ccal,E^u)}$ and therefore, by virtue of the commutative diagram above, of a well defined object $B_t$ in $\algdef_{\uend_\ccal (E^u)}(\kfor)$ corresponding to a formal deformation of the $\eo$-algebra $B=\uend_\ccal(E^u)$. 
By commutativity of the diagram, we have the equality of connected components $[\theta\circ \M (B_t)]=[\alpha]$ in $\pi_0\catdefa_\ccal(\kfor)$. 

\end{proof}

\appendix


\section{Deformations as !-group actions}\label{ssdefindcoh}

As a consequence of Corollary \ref{corassfmp}, we give a description of the formal $\et$-moduli problem $\catdefa_\ccal$ for $\ccal$ a $k$-linear $\oo$-category over a field $k$ in terms of $!$-actions of the formal group $\Omega\spf(A)$ on $\ccal$.

\begin{nota}\label{notaaut}
Let $k$ be a field and let $\ccal$ be a $k$-linear $\oo$-category, we set the notations $\aut_\ccal=\Omega\catdef_\ccal$ and $\auts_\ccal=\Omega\catdefa_\ccal$. The natural transformation $\theta_\ccal:\catdef_\ccal\mo \catdefa_\ccal$ induces a natural transformation $\aut_\ccal\mo \auts_\ccal$. 
\end{nota}

\begin{rema}\label{remautcg}
Let $k$ be a field and let $\ccal$ be a compactly generated $k$-linear $\oo$-category. Because the embedding $\catdefc_\ccal\mo \catdef_\ccal$ of compactly generated deformations into all deformations is fully faithful, the induced map $\Omega\catdefc_\ccal\mo \Omega\catdef_\ccal=\aut_\ccal$ is a homotopy equivalence. 
\end{rema}

\begin{rema}\label{remaauts} Let $k$ be a field and let $\ccal$ be a $k$-linear $\oo$-category. By \cite[Cor 5.3.8]{dagX} the functor $\catdef_\ccal$ is a $2$-proximate formal $\et$-moduli problem, which implies by Remark \ref{remaproxom} that the functor $\aut_\ccal$ is a $1$-proximate formal $\et$-moduli problem and that the induced map $\Omega \aut_\ccal\mo \Omega \auts_\ccal$ is an equivalence. By Corollary \ref{corassfmp}, we deduce that the induced map $L_2(\aut_\ccal)\mo \auts_\ccal$ is an equivalence. 

If in addition $\ccal$ is tamely compactly generated, then by \cite[Prop 5.3.21]{dagX} we have that $\catdefc_\ccal$ is $1$-proximate and we deduce from Remark \ref{remautcg} that $\aut_\ccal$ is a formal $\et$-moduli problem and that the natural map $\aut_\ccal\mo \auts_\ccal$ is an equivalence. 
\end{rema}

\begin{rema}\label{remaautcequiv}
Let $k$ be a field and let $\ccal$ be a $k$-linear $\oo$-category. It can be shown that there exists a natural transformation $\aut_\ccal\mo \objdefet_{\id_\ccal}$ which is an equivalence, where $\id_\ccal$ is seen as an object of the $k$-linear $\oo$-category $\End_k(\ccal)$ of $k$-linear endofunctors of $\ccal$. Having this, we see that for every artinian $\et$-algebra $A$ the space $\aut_\ccal(A)$ parametrizes pairs $(F_A,u)$ with $F_A$ an object of the $\oo$-category $\End(\ccal)\te_k \Mod_A \simeq \funl_A(\ccal\te_k A, \ccal\te_k A)$ (which is necessarily an $A$-linear autoequivalence of $\ccal\te_k A$) and $u:F_A\te_A k \simeq \id_\ccal$ an equivalence in $\End(\ccal)$. 
Moreover the space $\auts_\ccal (A)$ is the classifying space of right Ind-coherent $A$-modules in $\End(\ccal)$ which deforms $\id_\ccal$, or in other words of pairs $(F_A, u)$ with $F_A$ an object in $\End(\ccal)\te_k \Mod^!_A$ and $u:F_A\te_A k\simeq \id_\ccal$ an equivalence. 
\end{rema}

\begin{cor}\label{corcatdefsactions}
Let $k$ be a field and $\ccal$ a $k$-linear $\infty$-category. Then the formal $\et$-moduli problem $\catdefa_\ccal$ is given by the formula
$$\catdefa_\ccal(A)\simeq \map_{\mongp_{\eo}(\fmpet)} (\Omega \spf(A), \auts_\ccal).$$
If moreover $\ccal$ is tamely compactly generated, then the formal $\et$-moduli problem $\catdefa_\ccal$ is given by the formula
$$\catdefa_\ccal(A)\simeq \map_{\mongp_{\eo}(\fmpet)} (\Omega \spf(A), \aut_\ccal).$$
\end{cor}

\begin{proof} 
The first assertion follows from the definition of $\auts_\ccal$ and from Proposition \ref{propfmpgp}. The second follows from the first combined with Remark \ref{remaauts}. 
\end{proof}

\begin{const} (Ind-coherent sheaves on formal $\eo$-moduli problems as a symmetric monoidal functor). 
Let $k$ be a field. Consider the cocartesian fibration $\rmod^!(\modk)\mo \algkart$ of \cite[Const 3.4.11]{dagX} which is classified by the functor $\nu:\algkart \mo \catbig$ given by $\nu(A)=\rmod^!_A$ and functoriality for a morphism $f:A\mo B$ in $\algkart$ is given by the $!$-pullback $f^!:\Mod^!_A\mo \Mod^!_B$. The functor $f^!$ preserves small colimits. It follows that $\nu$ factors through the subcategory $\prl\subseteq \catbig$ of presentable $\oo$-categories with functors that are left adjoints. The full subcategory $\algkart\subseteq \algk$ is stable under tensor product and therefore inherits a symmetric monoidal structure $\algkartte$ from $\algk^\te$. A variant of \cite[Thm 4.8.5.16]{ha} allows us to promote the functor $\nu$ to a symmetric monoidal functor $\nu^\te: \algkartte \mo \prlte$. 

Recall that the Yoneda embedding $j:\algkart\mo \fun_*(\algkart, \spaces)^{op}$ can be promoted to a symmetric monoidal functor $j^\te: \algkartte\mo \fun_*(\algkart, \spaces)^{op, \times}$ where we endow the $\oo$-category $\fun_*(\algkart, \spaces)$ with its cartesian monoidal structure. Applying the monoidal variant of \cite[Thm 5.1.5.6]{htt}, we deduce that $\nu^\te$ admits an essentially unique factorization as a composition 
$$\algkartte\lmos{j^\te}\fun_*(\algkart, \spaces)^{op, \times}  \lmos{(\rindcohte, (-)^!)}\prlte$$
such that the functor $(\rindcoh, (-)^!)$ preserves small limits (recall that a limit in $\prl$ can be computed as the limit of the corresponding diagram in $\catbig$). For any functor $X:\calgkart\mo \spaces$, the $\oo$-category $\rindcoh(X)$ is the $k$-linear $\oo$-category of right Ind-coherent sheaves on $X$ and the symmetric monoidal structure expresses the fact that for any two functors $X,Y:\algkart\mo \spaces$, the natural functor $\indcoh(X)\te_k \rindcoh(Y)\mo \rindcoh(X\times Y) $ is an equivalence in $\prl$. Similarly we can define left Ind-coherent sheaves. 

Passing to right adjoints (through the symmetric monoidal equivalence $(\prl)^{op, \te}\simeq \prrte$), we obtain a symmetric monoidal functor 
$$\fun_*(\algkart, \spaces)^{\times} \lmos{(\rindcohte, (-)_*)} \prrte$$
where the functoriality is now given by the pushforward. 
\end{const} 

Consider the lax symmetric monoidal inclusion $\prrte\subseteq \catbigte$. 

\begin{lem}\label{lemindcohpush}
Let $k$ be a field. The restriction of the composition  
$$\fun_*(\algkart, \spaces)^{\times} \lmos{(\rindcohte, (-)_*)} \prrte\subseteq \catbig^\times$$
to the symmetric monoidal full subcategory $\eofmpte$ spanned by formal $\eu$-moduli problems factors through the lax symmetric monoidal inclusion $\prlte\subseteq \catbig^\times$ and is symmetric monoidal, providing a symmetric monoidal functor 
$$(\rindcohte, (-)_*): \eofmpte \mo \prlte.$$
\end{lem} 

\begin{proof} 
Let $f:X\mo Y$ be a map of $\eo$-formal moduli problems over $k$ and let $f_*:\rindcoh(X)\mo \rindcoh(Y)$ be the induced functor. We wish to show that $f_*$ admits a right adjoint, or equivalently by the adjoint functor theorem \cite[Cor 5.5.2.9]{htt} that it preserves small colimits. Let $A$ and $B$ be the augmented $\eo$-algebras over $k$ corresponding to $X$ and $Y$ respectively through the equivalence $\fmpeo\mo \algka$. Let $\varphi:A\mo B$ be the map corresponding to $f$. By the Koszul duality for modules \cite[Prop 3.5.1]{dagX} there exists equivalences $\rindcoh(X)\simeq \lmod_A$ and $\rindcoh(Y)\simeq \lmod_B$ and the induced diagram 
$$\xymatrix{ \rindcoh(X) \ar[r]^-{f_* } \ar[d]^-{\wr } & \rindcoh(Y)\ar[d]^-{\wr } \\ \lmod_A  \ar[r]^-{\varphi^* } & \lmod_B }$$
commutes up to homotopy, where $\varphi^*$ is the left adjoint to the forgetful functor $\varphi_*:\lmod_B \mo \lmod_A$ along $\varphi$. The functor $\alpha^*$ preserves small colimits which implies that $f_*$ has the same property. 
The second assertion follows directly from the existence of the symmetric monoidal functor $(\rindcohte, (-)^!) : (\eofmpte)^{op}\mo \prlkte$. 
\end{proof}

\begin{nota}
The symmetric monoidal functor $(\rindcohte, (-)_*)$ provided by Lemma \ref{lemindcohpush} will be simply denoted by $\rindcohte$. Consider the unit $\spec(k)$ in $\eofmpte$ as a commutative algebra object. We have an equivalence of commutative algebra objects $\rindcoh(\spec(k))\simeq \modk$ in $\prlte$. Therefore $\indcohte$ induces a symmetric monoidal functor 
$$\eofmpte\simeq \Mod_{\spec(k)} (\eofmp)^\times \lmos{\indcohte} \Mod_{\modk^\te} (\prl)^\te =\prlkte$$
into the symmetric monoidal $\oo$-category of presentable $k$-linear $\oo$-categories. 

For any $\eo$-formal moduli problem $X$ over $k$, the space $\map_{\eofmp} (\spec(k), X)\simeq X(k)$ is contractible, which implies that the projection $(\eofmp)_{\spec(k)/}\mo \fmp$ is an equivalence of $\oo$-categories. Because $\spec(k)$ and $\modk$ are the units of the symmetric monoidal $\oo$-categories $\eofmpte$ and $\prlkte$ respectively, we have natural symmetric monoidal structures $((\eofmp)_{\spec(k)/})^\times$ and $((\prlk)_{\modk/})^\te$ on the slice $\oo$-categories (see \cite[Rem 2.2.2.5]{ha}). Moreover the symmetric monoidal functor $\rindcohte$ induces a symmetric monoidal functor 
$$\eofmpte \simeq ((\eofmp)_{\spec(k)/})^\times\lmos{\rindcohtep} ((\prlk)_{\modk/})^\te.$$
The data of an object of the $\oo$-category $(\prlk)_{\modk/}$ is essentially given by a pair $(\ccal, E)$ with $\ccal$ a $k$-linear $\oo$-category and $E$ an object of $\ccal$. The functor $\rindcohp$ sends a formal moduli problem $X$ to the pair $(\rindcoh(X), k)$ where $\indcoh(X)$ is the $k$-linear $\oo$-category of Ind-coherent sheaves on $X$ and $k$ is the skyscraper sheaf with value $k$ at the $k$-point of $X$. If $A$ is the augmented $\eo$-algebra of $X$, the sheaf $k$ corresponds to the left $A$-module $A$ through the equivalence $\rindcoh(X)\simeq \lmod_A$.  
\end{nota}

We have the following consequence of \cite[Thm 4.8.5.11]{ha}.

\begin{prop}\label{propadj}
Let $k$ be a field. The symmetric monoidal functor 
$$\rindcohtep: \eofmpte \mo ((\prlk)_{\modk/})^\te$$ 
admits a lax symmetric monoidal right adjoint which sends a pair $(\ccal, E)$ to the formal $\eo$-moduli problem $\objdefa_{(\ccal, E)}$ of deformations of $E$ in $\ccal$. 
\end{prop} 

\begin{proof} 
By \cite[Cor 7.3.2.7]{ha}, it suffices to show that the underlying functor $\rindcohp$ of $\rindcohtep$ admits a right adjoint. By Koszul duality for modules \cite[Prop 3.5.1]{dagX}, there exists a homotopy commutative diagram 
$$\xymatrix{  \eofmp \ar[rr]^-{\rindcohp } \ar[d]_-{\tone } && \prlkp \\ \algka  \ar[rru]_-{\Theta_* } }$$
where $\tone$ is the equivalence of \cite[Thm 3.0.4]{dagX} and $\Theta_*$ is the composition of the functor constructed in \cite[ Const 4.8.3.24]{ha} with the forgetful functor $\algka\mo \algk$. Therefore $\Theta_*$ is given on objects by $\Theta_*(A)=(\lmod_A, A)$. By \cite[Thm 4.8.5.11]{ha}, $\Theta_*$ has a right adjoint which sends a pair $(\ccal, E)$ to the augmented $\eo$-algebra $k\oplus \uend_\ccal (E)$, with $\uend_\ccal (E)$ the $\eo$-algebra of endomorphisms of $E$ in $\ccal$. This implies that $\rindcohp$ has a right adjoint, which by \cite[Thm 5.2.8]{dagX} recalled in §\ref{ssobjdef}, is given on objects by $(\ccal, E)\mapsto\objdefa_{(\ccal, E)}$. 
\end{proof}

\begin{rema}\label{remadjobjdef}
Let $k$ be a field. Consider the restriction functor $\mrm{Rest}: \fmpeo\mo \fmpet$ which admits a left adjoint $\mrm{Ext}: \fmpet\mo \fmpeo$ given by the left Kan extension along the forgetful functor $\algetkart\mo \algkart$. By composition we obtain from Lemma \ref{propadj} that the symmetric monoidal functor 
$$\rindcohtep: \fmpette \mo ((\prlk)_{\modk/})^\te$$
admits as right adjoint the lax symmetric monoidal functor 
$$\objdefeta : ((\prlk)_{\modk/})^\te \mo \fmpette$$
where $\objdefeta_{(\ccal, E)}$ is the restriction of the formal $\eo$-moduli problem $\objdefa_{(\ccal, E)}$ on $\et$-algebras. Passing to algebra objects, we obtain a functor 
$$\alg(\rindcohtep) : \mon_{\eo}(\fmpet) \mo \alg(\prlk)$$
with right adjoint given by the functor $\objdefeta$ sending a pair $(\dal^\te, \unit)$ where $\dal^\te$ is a monoidal $k$-linear $\oo$-category and $\unit$ its unit object to the formal moduli problem $\objdefeta_{(\dal, \unit)}$ of deformations of $\unit$ in $\dal$. 
\end{rema}

\begin{cor}\label{corsgroupact}
Let $k$ be a field and let $\ccal$ be a $k$-linear $\oo$-category. Then the formal $\et$-moduli problem $\catdefa_\ccal$ is given by the formula
$$\catdefa_\ccal (A) \simeq \fun^\te_k (\rindcohte(\Omega \spf(A)), \End(\ccal)^\te)^\simeq .$$
In other words $\catdefa_\ccal(A)$ is naturally homotopy equivalent to the space 
$$\lmod_{\rindcohte(\Omega\spf(A))} (\prlk)^\simeq  \times_{(\prlk)^\simeq} \{\ccal\}$$
of left actions of the monoidal $\oo$-category $\rindcohte(\Omega\spf(A))$ on $\ccal$. 
\end{cor}

\begin{proof} 
By Corollary \ref{corcatdefsactions}, for $A$ an artinian $\et$-algebra, there is a natural homotopy equivalence 
$$\catdefa_\ccal (A)\simeq\map_{\mongp_{\eo}(\fmpet)} (\Omega \spf(A), \auts_\ccal).$$
Besides, there exists an equivalence of functors $\aut_\ccal\simeq \objdefet_{\id_\ccal}$ which induces a natural equivalence $\auts_\ccal\simeq \objdefeta_{\id_\ccal}$. Applying the adjunction of Remark \ref{remadjobjdef}, we obtain a natural homotopy equivalence 
$$\map_{\mongp_{\eo}(\fmpet)} (\Omega \spf(A), \auts_\ccal)\simeq \fun^\te_k (\rindcohte(\Omega \spf(A)), \End(\ccal)^\te)^\simeq$$
which proves our claim. 
\end{proof}


\vspace{2ex}
\bibliographystyle{amsalpha}
\begin{small}
\bibliography{/users/Anthony/Dropbox/MATH/Latex/Bibliographie/ref}
\end{small}
\vspace{5.mm}

  \begin{tabular}{l}
\small{\textsc{Anthony Blanc}} \\
   \hspace{.1in} \small{\textsc{Max Planck Institute for Mathematics, 53111 Bonn, Deutschland}}\\
   \hspace{.1in} \small{\textsc{Email}}: {\bf ablanc@mpim-bonn.mpg.de} \\
  \end{tabular}
  
  \vspace{2.mm}

  \begin{tabular}{l}
   \small{\textsc{Ludmil Katzarkov}} \\
   \hspace{.1in} \small{\textsc{Universit\"at Wien, Fakult\"at f\"ur Mathematik,  1090 Wien, \"Osterreich }}\\
   \hspace{.1in} \small{\textsc{Laboratory of Mirror Symmetry NRU HSE, Moscow, Russia}}\\
   \hspace{.1in} \small{\textsc{Email}}: {\bf lkatzarkov@gmail.com} \\
  \end{tabular}

\vspace{2.mm}

  \begin{tabular}{l}
\small{\textsc{Pranav Pandit}} \\
   \hspace{.1in} \small{\textsc{Universit\"at Wien, Fakult\"at f\"ur Mathematik,  1090 Wien, \"Osterreich}}\\
   \hspace{.1in} \small{\textsc{Email}}: {\bf pranav.pandit@univie.ac.at} \\
  \end{tabular}

\end{document}